\documentclass[12pt,regno]{amsart}
\usepackage{epsf,epsfig,amsmath}
\usepackage{amssymb,latexsym}
\usepackage{tikz}
\usepackage{tikz-cd}
\usepackage{amsthm} 
\usepackage{amsfonts}
\usepackage{graphicx,psfrag}

\numberwithin{equation}{section}

\newtheorem{thm}{Theorem}[section]
\newtheorem{pro}[thm]{Proposition}
\newtheorem{lem}[thm]{Lemma}
\newtheorem{con}[thm]{Conjecture}
\newtheorem{cor}[thm]{Corollary}
\newtheorem*{cor*}{Corollary}

\theoremstyle{remark}\newtheorem{rem}[thm]{Remark}

\theoremstyle{definition}\newtheorem{defi}[thm]{Definition}
\theoremstyle{definition}\newtheorem{exa}[thm]{Example}
\theoremstyle{definition}\newtheorem{probl}[thm]{Problem}

\newcommand\blfootnote[1]{%
  \begingroup
  \renewcommand\thefootnote{}\footnote{#1}%
  \addtocounter{footnote}{-1}%
  \endgroup
}

\title[Kazhdan--Lusztig $R$-polynomials for pircons]{Kazhdan--Lusztig $R$-polynomials for pircons}

\author{Mario Marietti}

\address{Dipartimento  di Ingegneria Industriale e Scienze Matematiche, Universit\`a Politecnica delle Marche, Via Brecce Bianche, 60131 Ancona,  Italy}

\email{m.marietti@univpm.it}


\begin{document}

\begin{abstract}
The purpose of this work is to provide a common combinatorial framework for some of the analogues and generalizations of Kazhdan--Lusztig $R$-polynomials that have appeared since the introduction of these remarkable polynomials (e.g., parabolic Kazhdan--Lusztig $R$-polynomials, Kazhdan--Lusztig $R$-polynomials of zircons, and Kazhdan--Lusztig--Vogan polynomials for fixed point free involutions). 
\end{abstract}

\maketitle\blfootnote{\copyright \quad 2019. This manuscript version is made available under the CC-BY-NC-ND 4.0 license http://creativecommons.org/licenses/by-nc-nd/4.0/\\ The final publication is available at: https://doi.org/10.1016/j.jalgebra.2019.05.038}

\section{Introduction}
Kazhdan--Lusztig $R$-polynomials play a  central role in Lie theory and representation theory. They are polynomials $R_{u,v}(q)$, in one variable $q$, which are associated with pairs of elements $u,v$ in a Coxeter group $W$.  As well as the celebrated Kazhdan--Lusztig $P$-polynomials, they were introduced by Kazhdan and Lusztig in \cite{K-L} in order to study the (now called)  
Kazhdan--Lusztig representations of the Hecke algebra of $W$. These polynomials  have soon found applications in many other contexts. 

Since their introduction, Kazhdan--Lusztig $R$-polynomials have been studied a lot. There is an enormous literature on these polynomials and related objects. In particular, there are several works on analogues and generalizations of Kazhdan--Lusztig $R$-polynomials.  The purpose of this work is to provide a common combinatorial framework for some of these analogues and generalizations. For example,  parabolic Kazhdan--Lusztig $R$-polynomials (see \cite{Deo87}), Kazhdan--Lusztig $R$-polynomials of diamonds and zircons (see \cite{BCM2} and \cite{MJaco}, respectively), and Kazhdan--Lusztig--Vogan $R$-polynomials for fixed point free involutions (which are the Kazhdan--Lusztig--Vogan $R$-polynomials associated with the action of $Sp(2n,C)$  on the flag variety of $SL(2n,C)$, see \cite{AH}) lie within this framework.

Let us present this general setting. In \cite{AHH}, Abdallah, Hansson, and A. Hultman introduce a generalization of the concept of a zircon, which they call a pircon since pircons relate to special partial matchings in the same way as zircons relate to special matchings. A pircon is a partially ordered set $P$ such that, for every non minimal element $x\in P$, the subposet $\{y\in P:y\leq x\}$, denoted $P_{\leq x}$, is finite and admits a special partial matching. Special partial matchings are a generalization of special matchings (see Section~2) and were introduced in \cite{AH} in order to study the Kazhdan--Lusztig--Vogan $R$-polynomials for fixed point free involutions. For each non minimal element $v\in P$, we may fix a special partial matching $M_v$ of $P_{\leq v}$ and call the pair $(P, \{M_v: v \text{ non minimal}\})$ a refined pircon. For each refined pircon, we define two families $\{R^{q}_{u,v\in P}(q)\}$ and $\{R^{-1}_{u,v\in P}(q)\}$ of polynomials in one variable $q$, which are associated with pairs of elements $u,v\in P$. We call these polynomials the \emph{Kazhdan--Lusztig $R$-polynomials} of the refined pircon.  
Note that, in general, Kazhdan--Lusztig $R$-polynomials do depend on the refinement of the pircon $P$. When $P$ is a refined zircon, a diamond, or a Coxeter group (partially ordered by Bruhat order), the two families $\{R^{q}_{u,v\in P}(q)\}$ and $\{R^{-1}_{u,v\in P}(q)\}$ coincide, and are actually the Kazhdan--Lusztig $R$-polynomials of $P$ as a refined zircon, a diamond, or a Coxeter group, respectively. When $P$ is the set of minimal coset representatives of left cosets of a parabolic subgroup $W_J$ in a Coxeter group $W$,  the two families $\{R^{q}_{u,v\in P}(q)\}$ and $\{R^{-1}_{u,v\in P}(q)\}$ are the parabolic Kazhdan--Lusztig $R$-polynomials of type $q$ and $-1$, respectively.  When $P$ is the set of twisted identities of the symmetric group $S_{2n}$, the two families $\{R^{q}_{u,v\in P}(q)\}$ and $\{R^{-1}_{u,v\in P}(q)\}$ are the Kazhdan--Lusztig $R$-polynomials and $Q$-polynomials for fixed point free involutions.

\section{Notation, definitions and preliminaries}
\label{preliminari}

This section reviews the background material that is needed  in the rest of this work.
We  follow   \cite[Chapter 3]{StaEC1} and  \cite{BB} for undefined notation and 
terminology concerning, respectively,  partially ordered sets and Coxeter groups.

\subsection{Special matchings and special partial matchings.}
Let $P$ be a partially ordered set (poset for short). An element $y\in P$ {\em covers} $x\in P$  if the interval $[x,y]$ coincides with $\{x,y\}$; in this case, we write $x \lhd y$ as well as $y \rhd x$.
If $P$ has a minimum (respectively, a maximum), we denote it by $\hat{0}_P$  (respectively, $\hat{1}_P$). 
The poset $P$ is {\em graded} if $P$ has a minimum and there is a function
$\rho : P \rightarrow {\mathbb N}$ (the {\em rank function}
of $P$) such that $\rho (\hat{0})=0$ and $\rho (y) =\rho (x)
+1$ for all $x,y \in P$ with $x \lhd y$. 
(This definition is slightly different from the one given in \cite{StaEC1}, but is
more convenient for our purposes.) 
The {\em Hasse diagram} of $P$ is any drawing of the graph having $P$ as vertex set and $ \{ \{ x,y \} \in \binom {P}{2} \colon\,  \text{ either $x \lhd y$ or $y \lhd x$} \}$ as edge set, with the convention that, if $x \lhd y$, then the edge $\{x,y\}$ goes upward from $x$ to $y$. When no confusion arises, we make no distinction between the Hasse diagram and its underlying graph.

A {\em matching} of a poset \( P \) is an involution
\( M:P\rightarrow P \) such that \( \{v,M(v)\}\) is an edge in the Hasse diagram of $P$, for all \( v\in V \).
A matching \( M \) of \( P \) is {\em special} if\[
u\lhd v\Longrightarrow M(u)\leq M(v),\]
 for all \( u,v\in P \) such that \( M(u)\neq v \). 

The following definitions are taken from \cite{MJaco}, \cite{AH}, and \cite{AHH}, respectively. Given a poset $P$ and $x\in P$, we set $P_{\leq x}=\{y \in P \mid y \leq x\}$.
\begin{defi}
\label{zircone}
A poset $Z$ is a \emph{zircon} provided that, for every non-minimal element $x \in Z$, the order ideal $P_{\leq x}$ is finite and admits a special matching.
\end{defi}

\begin{defi}\label{accoppiamento parziale}
Let $P$ be a finite poset with $\hat1_P$.
A \emph{special partial matching} of $P$ is an involution $M: P \to P$ such that
\begin{itemize}
  \item $M(\hat1) \lhd \hat1$,
  \item for all $x\in P$, we have $M(x) \lhd x$, $M(x)=x$, or $x\lhd M(x)$, and
  \item if $x\lhd y$ and $M(x) \neq y$, then $M(x)<M(y)$.
\end{itemize}
\end{defi}
\begin{defi}
A poset $P$ is a \emph{pircon} provided that, for every non-minimal element $x \in P$, the order ideal $P_{\leq x}$ is finite and admits a special partial matching.
\end{defi}
The terminology comes from the fact that a special partial matching without fixed points is precisely a special matching and the fact that pircons relate to special partial matchings in the same way as zircons relate to special matchings. Connected zircons and pircons are graded posets (the argument for the zircons in  \cite[Proposition 2.3]{H08} applies also to pircons). We always denote the rank function by $\rho$.

Given a poset $P$ and $w\in P$, we say that $M$ is a matching of $w$ if $M$ is a matching of $P_{\leq w}$, and we denote by $SPM_w$ the set of all special partial matchings of $w$. Hence, if $P$ is a pircon then $SPM_w\neq \emptyset$ for all $w\in P\setminus \{\hat{0}_P\} $. 
In pictures, we visualize a special partial matching $M$ of a poset $P$ by taking the Hasse diagram of $P$ and coloring in the same way, for all $x\in P$, either the edge $\{x,M(x)\}$ if $M(x)\neq x$, or a circle around $x$ if $M(x)=x$.
An example is in Figure~\ref{esempio}, where the special partial matching is colored in dashed black and fixes only the bottom element $\hat{0}_P$. 
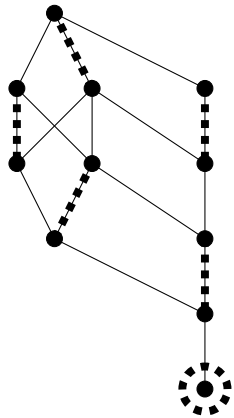
\begin{figure}[h]
\begin{center}
$$
\begin{tikzpicture}

\draw[fill=black]{(0,0) circle(3pt)};
\draw[fill=black]{(-0.5,1) circle(3pt)};
\draw[fill=black]{(-0.5,2) circle(3pt)};
\draw[fill=black]{(0,3) circle(3pt)};
\draw[fill=black]{(0.5,1) circle(3pt)};
\draw[fill=black]{(0.5,2) circle(3pt)};
\draw[dashed, line width=3pt]{(-0.5,2)--(-0.5,1)};
 \draw[dashed, line width=3pt]{(0,0)--(0.5,1)}; 
 \draw[dashed, line width=3pt]{(0.5,2)--(0,3)}; 
 \draw{(0.5,2)--(0.5,1)};
 \draw{(0,0)--(-0.5,1)}; 
 \draw{(-0.5,2)--(0,3)}; 
 \draw{(-0.5,1)--(0.5,2)}; 
 \draw{(0.5,1)--(-0.5,2)}; 
  \draw{(0.5,2)--(0,3)}; 

  \draw{(-0.5,1)--(-0.5,2)}; 
  \draw{(0.5,1)--(0,0)}; 
 
 \draw[fill=black]{(2,1) circle(3pt)};
\draw[fill=black]{(2,2) circle(3pt)};
\draw[fill=black]{(2,0) circle(3pt)};
\draw[fill=black]{(2,-1) circle(3pt)};
\draw[fill=black]{(2,-2) circle(3pt)};
  \draw{(0,3)--(2,2)}; 
   \draw{(0.5,2)--(2,1)}; 
    \draw{(0.5,1)--(2,0)}; 
     \draw{(0,0)--(2,-1)};

\draw[dashed, line width=3pt]{(2,-2) circle (.3)};
\draw[dashed, line width=3pt]{(2,2)--(2,1)};
\draw[dashed, line width=3pt]{(2,0)--(2,-1)};
  \draw{(2,2)--(2,-2)}; 
\end{tikzpicture}$$
\end{center}
\caption{\label{esempio} A special partial matching.} 
\end{figure}

For a proof of the following result, see \cite[Lemma 5.2]{AHH}.
\begin{lem}[Lifting property of special partial matchings]
\label{lifting}
Let $P$ be a finite poset with $\hat1_P$, and $M$ be a special partial matching of $P$. If $x,y \in P$ with $x<y$ and $M(y) \leq y$, then
\begin{itemize}
  \item[(i)]   $M(x) \leq y$,
  \item[(ii)]  $M(x) \leq x \implies M(x)<M(y)$, and
  \item[(iii)] $M(x) \geq x \implies x \leq M(y)$.
\end{itemize}
\end{lem}

Given a finite poset $P$, we denote by $\Delta(P)$ its order complex, which is the simplicial complex whose faces are the chains in $P$. As usual, for $x,y \in P$, we write $\Delta(x,y)$ for $\Delta([x,y]\setminus \{x,y\})$. The following is Theorem 6.4 of \cite{AHH} (we refer the reader to \cite[Section 2.3]{AHH} for a brief account on PL topology).

\begin{thm} \label{intervalli in pirconi}
Let $P$ be a pircon and $x,y\in P$ with $x<y$. Then $\Delta(x,y)$ is a PL ball or a PL sphere.
Moreover, there exist $x,y\in P$, $x<y$,  such that $\Delta(x,y)$ is a ball if and only if $P$ is not a zircon.
\end{thm}

\subsection{Coxeter groups}
We fix our notation on a Coxeter system $(W,S)$ in the following list:
\smallskip 

{\renewcommand{\arraystretch}{1.2}
$
\begin{array}{@{\hskip-1.3pt}l@{\qquad}l}
m(s,t) &  \textrm{the entry of the Coxeter matrix of  $(W,S)$ in position $(s,t)\in S\times S$}, 
\\
e &  \textrm{identity of $W$}, 
\\
\ell  &  \textrm{the length function of $(W,S)$},
\\
T &= \{ w s w ^{-1} : w \in W, \; s \in S \},  \textrm{  the set of {\em reflections} of $W$},
\\
D_R(w) & =\{ s \in S : \; \ell(w  s) < \ell(w ) \},  \textrm{ the right descent set of $w\in W$},
\\
D_L(w) & =\{ s \in S : \; \ell( sw) < \ell(w ) \},  \textrm{ the left descent set of $w\in W$},
\\
W_J & \textrm{ the parabolic subgroup of $W$ generated by $J\subseteq S$},
\\
W^J &=\{ w \in W \, : \; D_{R}(w)\subseteq S\setminus J \}, \textrm{ the set of minimal left coset representatives},
\\
^{J} W &=\{ w \in W \, : \; D_{L}(w)\subseteq S\setminus J \}, 
\textrm{ the set of minimal right coset representatives},
\\
\leq & \textrm{ Bruhat order on $W$ (as well as any other order on a poset $P$)},
\\
\textrm{$[u,v]$} & =\{ w \in W \, : \; u \leq w \leq v \}, \textrm{ the (Bruhat) interval generated by $u,v\in W$},
\\
w_0(J) &  \textrm{ the unique maximal  element of $[e,w]\cap W_{J}$, for $J\subseteq  S$},
\\
w_0(s,t) & =w_0(\{s,t\}),  \textrm{  for $s,t\in S$},
\\
\textrm{ $[u,v]^{H}$} &  = \{ z \in W^{H}: \;  u \leq z \leq v \}, \textrm{ the parabolic (Bruhat)  interval generated by $u,v\in W^H$}.

\end{array}$
\bigskip

We make use of the symbol ``$\textrm{-}$''  to separate letters in a word in the alphabet $S$ when we want to stress the fact that we are considering the word rather than the element that such word represents.

If $w\in W$, then a {\em reduced expression} for $w$ is a word $s_1\textrm{-}s_2\textrm{-}\cdots \textrm{-}s_q$ such that $w=s_1s_2\cdots s_q$ and $\ell(w)=q$. When no confusion arises, we also write that $s_1s_2\cdots s_q$ is a reduced expression for $w$.

The Coxeter group $W$ is partially ordered by {\em Bruhat order} (see, for example,  \cite[Section~2.1]{BB} or \cite[Section~5.9]{Hum}), which we denote by $\leq$. The Bruhat order (sometimes also called {\em Bruhat-Chevalley order}) is the partial order whose  covering relation $\lhd$ is as follows: if $u,v \in W$, then $u \lhd v$ if and only if $u^{-1}v \in T$ and $\ell (u)=\ell (v)-1$. The Coxeter group $W$, partially ordered by Bruhat order, is a graded poset having $\ell$ as its rank function.

The following  well-known characterization of Bruhat order
is usually referred to as the {\em Subword Property} (see \cite[Section~2.2]{BB} or \cite[Section~5.10]{Hum}), and is used repeatedly in this work,
often without explicit mention. 
By a {\em subword} of a word $s_{1}\textrm{-}s_{2} \textrm{-} \cdots \textrm{-} s_{q}$, we mean
a word of the form
$s_{i_{1}}\textrm{-} s_{i_{2}}\textrm{-} \cdots \textrm{-} s_{i_{k}}$, where $1 \leq i_{1}< \cdots
< i_{k} \leq q$.

\begin{thm}[Subword Property]
\label{subword}
Let $u,w \in W$. Then the following are equivalent:
\begin{itemize}
\item $u \leq w$ in the Bruhat order,
\item  every reduced expression for $w$ has a subword that is 
a reduced expression for $u$,
\item there exists a  reduced expression for $w$ having a subword that is 
a reduced expression for $u$.
\end{itemize}
\end{thm}

The following results are well known (see, e.g.,  \cite[Theorem~1.1]{Deo77},  \cite[Proposition~2.2.7]{BB} or \cite[Proposition~5.9]{Hum} for the first one, and  \cite[Section~2.4]{BB} or \cite[Section~1.10]{Hum} for the second one).
\begin{lem}[Lifting Property]
\label{ll}
Let $s\in S$ and $u,w\in W$, with $u\leq w$. 
\begin{enumerate}
\item[-] If $s\in D_R(w)$ and $s\in D_R(u)$ then $us\leq ws$.
\item[-] If $s\notin D_R(w)$ and $s \notin D_R(u)$ then $us\leq ws$.
\item[-] If $s\in D_R(w)$ and $s\notin D_R(u)$ then $us\leq w$ and $u\leq ws$.
\end{enumerate}
Symmetrically,  left versions of the three statements hold.
\end{lem}
\begin{pro}
\label{fattorizzo}
Let $J \subseteq S$. 
\begin{enumerate}
\item[(i)] 
Every $w \in W$ has a unique factorization $w=w^{J} \cdot w_{J}$ 
with $w^{J} \in W^{J}$ and $w_{J} \in W_{J}$; for this factorization, $\ell(w)=\ell(w^{J})+\ell(w_{J})$.
\item[(ii)] Every $w \in W$ has a unique factorization $w=\, _{J} w\,  \cdot \,  ^{J}\! w$ 
with $_{J} w \in W_{J}$, $^{J} \! w \in \,  ^{J} W$; for this factorization, $\ell(w)=\ell(_{J} w )+\ell(^{J} \! w)$.
\end{enumerate}
\end{pro}

Note that $v\leq w$  implies  both $v^J\leq w^J$  and ${^J\! v}\leq \,{^J\!w}$.

The \emph{Hecke algebra} of W, denoted $\mathcal H(W)$,  is the $\mathbb Z[q^{\frac{1}{2}},q^{-\frac{1}{2}}]$-algebra generated by $\{T_s : s \in S\}$ subject to the braid relations 
$$\underbrace{\cdots T_sT_rT_s}_{\text{$m(s,r)$ terms}}= \underbrace{\cdots T_rT_sT_r}_{\text{$m(s,r)$ terms}}\quad \text{ for all $s,r \in S$} $$
and the quadratic relations
$$T^2_s=(q-1)T_s+q \quad \text{ for all $s \in S$}.$$
For $w\in W$, denote by $T_w$ the product $T_{s_1}T_{s_2}\cdots T_ {s_k}$, where $s_1 s_2\cdots s_k$ is a reduced expression for $w$. The element $T_w$ is independent from the chosen reduced expression. The Hecke algebra $\mathcal H(W)$ is the free $\mathbb Z[q^{\frac{1}{2}},q^{-\frac{1}{2}}]$-module having the set $\{T_w : w \in W\}$ as a basis and multiplication uniquely determined by
$$ T_sT_w =\left\{
\begin{array}{ll}
T_{sw}, & \text{ if $sw >w$,} \\
qT_{sw} + (q-1)T_w, & \text{  if  $sw <w$,}
\end{array}
\right.
$$
for all $w \in W$ and $s\in S$. 
\bigskip

Recall that, given $w\in W$, we say that $M$ is a matching of $w$ if $M$ is a matching of the lower Bruhat interval $[e,w]$.
For \( s\in D_{R}(w) \), we have  a matching 
\(\rho_{s}  \) of $w$  
defined by   \( \rho_{s}(u)=us \), for all $u \in [e,w]$. Symmetrically,  for \( s\in D_{L}(w)$, we have  a matching 
$\lambda_{s}$ of $w$ 
defined by    $\lambda_{s}(u)=su$, for all $u \in [e,w]$. By
the Lifting Property (Lemma~\ref{ll}),  such $\rho_s$ and $\lambda_s$ are special matchings of $w$. 
We call these matchings, respectively,  \emph{right} and \emph{left multiplication
matchings}.

We refer to \cite{CM1} and \cite{CM2}  for more details concerning special matchings of Coxeter systems.

\section{Orbits in pircons}
\label{orbits}
In this section, we provide some easy results on pircons that are needed later. 

We say that an interval $[u, v]$ in a poset $P$ is {\em dihedral} if it is isomorphic to an interval in a Coxeter group with  two Coxeter generators, ordered by Bruhat order (see Figure~\ref{diheinte}).
\begin{figure}[h]
 \setlength{\unitlength}{8mm}
\begin{center}
\begin{picture}(12,6)
\thicklines
\put(0,0){\line(0,1){1}}
\put(0,0){\circle*{0.15}}
\put(0,1){\circle*{0.15}}

\put(3,0){\line(1,1){1}}
\put(3,0){\line(-1,1){1}}
\put(4,1){\line(-1,1){1}}
\put(2,1){\line(1,1){1}}
\put(3,0){\circle*{0.15}}
\put(3,2){\circle*{0.15}}
\put(2,1){\circle*{0.15}}
\put(4,1){\circle*{0.15}}

\put(7,0){\line(1,1){1}}
\put(7,0){\line(-1,1){1}}
\put(8,1){\line(0,1){1.5}}
\put(6,1){\line(0,1){1.5}}
\put(6,2.5){\line(1,1){1}}
\put(8,2.5){\line(-1,1){1}}
\put(6,1){\line(4,3){2}}
\put(8,1){\line(-4,3){2}}
\put(7,0){\circle*{0.15}}
\put(8,1){\circle*{0.15}}
\put(6,1){\circle*{0.15}}
\put(6,2.5){\circle*{0.15}}
\put(8,2.5){\circle*{0.15}}
\put(7,3.5){\circle*{0.15}}

\put(11,0){\line(1,1){1}}
\put(11,0){\line(-1,1){1}}
\put(12,1){\line(0,1){3}}
\put(10,1){\line(0,1){3}}
\put(10,1){\line(4,3){2}}
\put(12,1){\line(-4,3){2}}
\put(10,2.5){\line(4,3){2}}
\put(12,2.5){\line(-4,3){2}}
\put(10,4){\line(1,1){1}}
\put(12,4){\line(-1,1){1}}
\put(11,0){\circle*{0.15}}
\put(12,1){\circle*{0.15}}
\put(10,1){\circle*{0.15}}
\put(12,2.5){\circle*{0.15}}
\put(10,2.5){\circle*{0.15}}
\put(10,4){\circle*{0.15}}
\put(12,4){\circle*{0.15}}
\put(11,5){\circle*{0.15}}

\end{picture}
\end{center}
\caption{\label{diheinte} Dihedral intervals of rank 1,2,3,4}
\end{figure}
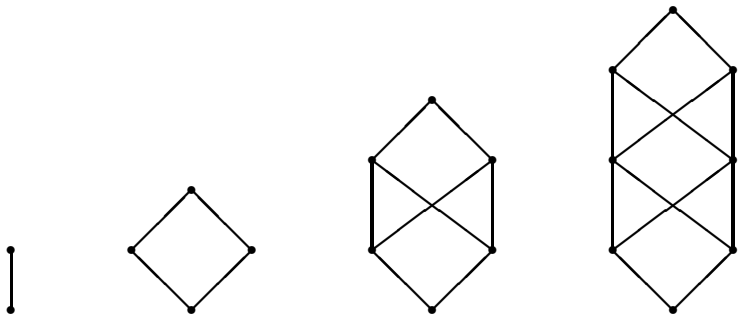

 The proof of the following result is straightforward from the Lifting Property of special partial matchings (Lemma \ref{lifting}).
\begin{lem}
\label{restringe}
Let $M$ be a special partial matching of a poset $P$, and $x,y \in P$, with $x\leq y$.  If $M(y)\leq y$ and $M(x) \geq x$, then $M$ restricts to a partial matching of the interval $[x,y]$ (which is special if and only if $M(y)\lhd y$).
\end{lem}

Given a pircon $P$ and two special partial matchings  $M$ and $N$ of $P$,  we denote by $\langle M,N \rangle $ the group of permutations of $P$ generated by $M$ and $N$. Furthermore, we denote by $\langle M,N \rangle (u)$ the orbit of an element $u\in P$ under the action of $\langle M,N \rangle $.

We say that an orbit $\mathcal O$ of the action of $\langle M,N \rangle $ is 
\begin{itemize}
\item \emph{dihedral}, if  $\mathcal O$ is isomorphic to a dihedral interval and $u\notin \{M(u),N(u)\}$ for all $u\in \mathcal O$,
\item \emph{chain-like}, if  $\mathcal O$ is isomorphic to a chain, $\hat{0}_{\mathcal O} \in \{M(\hat{0}_{\mathcal O}),N(\hat{0}_{\mathcal O})\}$, and $\hat{1}_{\mathcal O} \in \{M(\hat{1}_{\mathcal O}),N(\hat{1}_{\mathcal O})\}$.
\end{itemize}

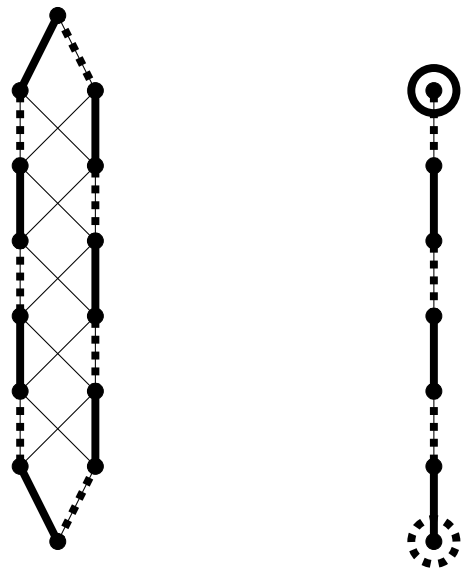
\begin{figure}[h]
\begin{center}
$$
\begin{tikzpicture}

\draw[fill=black]{(0,0) circle(3pt)};
\draw[fill=black]{(-0.5,1) circle(3pt)};
\draw[fill=black]{(-0.5,2) circle(3pt)};
\draw[fill=black]{(-0.5,3) circle(3pt)};
\draw[fill=black]{(0.5,1) circle(3pt)};
\draw[fill=black]{(0.5,2) circle(3pt)};
\draw[fill=black]{(0.5,3) circle(3pt)};
\draw[fill=black]{(-0.5,4) circle(3pt)};
\draw[fill=black]{(-0.5,5) circle(3pt)};
\draw[fill=black]{(0.5,4) circle(3pt)};
\draw[fill=black]{(0.5,5) circle(3pt)};
\draw[fill=black]{(0,7) circle(3pt)};
\draw[fill=black]{(-0.5,6) circle(3pt)};
\draw[fill=black]{(0.5,6) circle(3pt)};

\draw[dashed, line width=3pt]{(0,7)--(0.5,6)}; 
\draw[dashed, line width=3pt]{(-0.5,5)--(-0.5,6)}; 
\draw[dashed, line width=3pt]{(-0.5,4)--(-0.5,3)}; 
\draw[dashed, line width=3pt]{(-0.5,2)--(-0.5,1)};
 \draw[dashed, line width=3pt]{(0,0)--(0.5,1)}; 
 \draw[dashed, line width=3pt]{(0.5,2)--(0.5,3)}; 
\draw[dashed, line width=3pt]{(0.5,4)--(0.5,5)}; 
 
  \draw[line width=3pt]{(0,7)--(-0.5,6)}; 
 \draw[line width=3pt]{(0.5,5)--(0.5,6)}; 
\draw[line width=3pt]{(0.5,4)--(0.5,3)}; 
\draw[line width=3pt]{(0.5,2)--(0.5,1)};
 \draw[line width=3pt]{(0,0)--(-0.5,1)}; 
 \draw[line width=3pt]{(-0.5,2)--(-0.5,3)}; 
\draw[line width=3pt]{(-0.5,4)--(-0.5,5)};

 \draw{(-0.5,1)--(0.5,2)}; 
 \draw{(-0.5,3)--(0.5,2)}; 
 \draw{(-0.5,3)--(0.5,4)}; 
 \draw{(-0.5,5)--(0.5,4)}; 
 \draw{(0.5,1)--(-0.5,2)}; 
 \draw{(0.5,3)--(-0.5,2)}; 
 \draw{(0.5,3)--(-0.5,4)}; 
 \draw{(0.5,5)--(-0.5,4)}; 
  \draw{(0.5,6)--(-0.5,5)}; 
 \draw{(-0.5,6)--(0.5,5)}; 
 
 \draw{(0.5,2)--(0.5,3)}; 
 \draw{(0.5,5)--(0.5,4)}; 
  \draw{(-0.5,1)--(-0.5,2)}; 
 \draw{(-0.5,3)--(-0.5,4)}; 
  \draw{(-0.5,5)--(-0.5,6)}; 
  \draw{(0.5,1)--(0,0)}; 
 \draw{(0,7)--(0.5,6)}; 
 
 \draw[fill=black]{(5,1) circle(3pt)};
\draw[fill=black]{(5,2) circle(3pt)};
\draw[fill=black]{(5,3) circle(3pt)};
\draw[fill=black]{(5,4) circle(3pt)};
\draw[fill=black]{(5,0) circle(3pt)};
\draw[fill=black]{(5,6) circle(3pt)};
\draw[fill=black]{(5,5) circle(3pt)};

\draw[dashed, line width=3pt]{(5,0) circle (.3)};
 \draw[line width=3pt, line width=3pt]{(5,0)--(5,1)}; 
 \draw[line width=3pt, line width=3pt]{(5,2)--(5,3)}; 
  \draw[line width=3pt, line width=3pt]{(5,4)--(5,5)}; 
\draw[dashed, line width=3pt]{(5,2)--(5,1)};
\draw[dashed, line width=3pt]{(5,3)--(5,4)};
\draw[dashed, line width=3pt]{(5,5)--(5,6)};
\draw[line width=3pt]{(5,6) circle (.3)};
\draw{(5,0)--(5,6)}; 
\end{tikzpicture}$$
\end{center}
\caption{\label{orbitediedraliecatene} A rank 7 dihedral orbit  and a rank 6 chain-like orbit of  $\langle M,N \rangle $, where $M$ and $N$ are colored in solid and dashed black}
\end{figure}

Note that an orbit with two elements $w$ and  $N(w)=M(w)\neq w$ is dihedral, whereas an orbit with two elements $w=N(w)$ and  $M(w)=NM(w)\neq w$ is chain-like (see Figure~\ref{rango1}). This is the only case when a dihedral orbit and a chain-like orbit are isomorphic as posets.

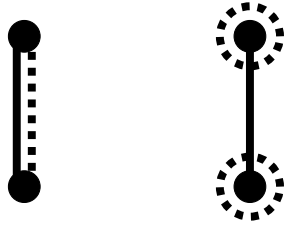
\begin{figure}[h]
\begin{center}
$$
\begin{tikzpicture}
 \draw[fill=black]{(2,2) circle(6pt)};
\draw[fill=black]{(2,0) circle(6pt)};
 \draw[line width=3pt]{(1.9,0)--(1.9,2)}; 
  \draw[line width=3pt, dashed]{(2.1,0)--(2.1,2)}; 
 \draw[fill=black]{(5,2) circle(6pt)};
\draw[fill=black]{(5,0) circle(6pt)};
\draw[dashed, line width=3pt]{(5,0) circle (.4)};
 \draw[line width=3pt]{(5,0)--(5,2)}; 
\draw[line width=3pt, dashed]{(5,2) circle (.4)};
\end{tikzpicture}$$
\end{center}
\caption{\label{rango1} A dihedral orbit of rank 1  and a chain-like orbit of rank 1}
\end{figure}

\begin{lem}
\label{noorbita}
Fix a pircon $P$.  Let $M$ and $N$ be two special partial matchings of an element $w$ in $P$. 
Every orbit $\mathcal O$ of $\langle M,N \rangle $ is either dihedral or chain-like. 
\end{lem}
\begin{proof}
Fix an orbit $\mathcal O$ and let $z$ be a maximal element in $\mathcal O$. 

Suppose $z\rhd M(z)$ and $z\rhd N(z)$. If $M(z)=N(z)$ then $\mathcal O$ is a dihedral orbit of rank 1; otherwise  $N(z)\rhd NM(z)$ and $M(z)\rhd MN(z)$ by the definition of a special partial matching. If $NM(z)=MN(z)$ then $\mathcal O$ is a dihedral orbit of rank 2; otherwise $MN(z)\rhd MNM(z)$ and $NM(z)\rhd NMN(z)$ by the definition of a special partial matching. We iterate this argument and obtain a dihedral orbit of rank $r$, where $r$ is the smallest number such that  $\underbrace{\cdots MNM}_{\text{$r$ letters}}(z)= \underbrace{\cdots NMN}_{\text{$r$ letters}}(z)$.

Suppose that one of the two matching fixes $z$, say $z= M(z)$. If $z= N(z)$ then $\mathcal O$ is a chain of rank 0; otherwise  $z\rhd N(z)$.  If $N(z)= MN(z)$ then $\mathcal O$ is a chain of rank 1; otherwise  $N(z)\rhd MN(z)$ by the definition of a special partial matching. 

If $MN(z)= NMN(z)$ then $\mathcal O$ is a chain of rank 2. Otherwise we claim that $MN(z) \rhd NMN(z)$. Indeed, suppose by contradiction that $MN(z) \lhd NMN(z)$: by the definition of a special partial matching $MNMN(z)\rhd N(z)$ and so  $MNMN(z)\rhd NMN(z)$. Again by the definition of a special partial matching, $NMNMN(z)\rhd MNMN(z)$ and so $NMNMN(z)\rhd z$. The last covering relation  is in contradiction with the fact that  $z$ is a maximal element in $\mathcal O$.

We iterate this argument and obtain a chain-like orbit of rank $r-1$, where $r$ is the smallest number such that  $\underbrace{\cdots MNM}_{\text{$r$ letters}}(z)= \underbrace{\cdots NMN}_{\text{$r$ letters}}(z)$.
\end{proof}

\begin{lem}
\label{tuttaorbita}
Fix a pircon $P$.  Let $M$ and $N$ be two special partial matchings of an element $w$ in $P$. 
Every orbit $\mathcal O$ of $\langle M,N \rangle $ is an interval in $P$, i.e. $\mathcal O=[\hat{0}_{\mathcal O},\hat{1}_{\mathcal O}]$.
\end{lem}
\begin{proof}
Fix an orbit $\mathcal O$ and, by contradiction, suppose $[\hat{0}_{\mathcal O},\hat{1}_{\mathcal O}]\setminus \mathcal O \neq \emptyset$. Let $z\in  [\hat{0}_{\mathcal O},\hat{1}_{\mathcal O}]\setminus \mathcal O$ be an element that is covered by some element, say $c$,  in $\mathcal O$. Since $c\neq \hat{0}_{\mathcal O}$, Lemma~\ref{noorbita} implies that $c$  must go down with one of the two matchings, say $M$, and the definition of a special partial matching implies $M(z)\lhd z$ and $M(z)\lhd  M(c)$. If $M(c) \neq \hat{0}_{\mathcal O}$ then $NM(c)\lhd M(c)$ and the definition of a special partial matching implies $NM(z)\lhd M(z)$ and  $NM(z)\lhd NM(c)$. By iteration, we find a number $k>1$ and a chain $z=z_0\rhd z_1=M(z)\rhd z_2=NM(z)\rhd \cdots \rhd z_{k-1}=\underbrace{\cdots MNM}_{\text{$(k-1)$ letters}}(z)\rhd z_k=\underbrace{\cdots MNM}_{\text{$k$ letters}}(z)$ with $z_i \in \{M(z_{i-1}), N(z_{i-1})\}$ for all $i\in [1,k]$, $z_{k-1}\in [\hat{0}_{\mathcal O},\hat{1}_{\mathcal O}]$,  and $z_{k}\notin [\hat{0}_{\mathcal O},\hat{1}_{\mathcal O}]$. This is a contradiction since both $M$ and $N$ restrict to matchings of  $[\hat{0}_{\mathcal O},\hat{1}_{\mathcal O}]$ by Lemma~\ref{restringe}.
\end{proof}
Notice that  Lemma~\ref{tuttaorbita} for  $\mathcal O$  dihedral (whose order complex is well known to be a PL sphere) follows also  by Theorem~\ref{intervalli in pirconi}, since a PL sphere can be a proper subcomplex neither of another PL sphere of the same dimension, nor of a PL ball of the same dimension.

\section{Hecke algebra actions}
\label{quattro}

In this section, we introduce certain representations of  Hecke algebras with two generators, which are needed in Section~\ref{KLMMpoly}.
 
Let $W_{s,r;d}$ be the dihedral Coxeter group with Coxeter generators $s$ and $r$ subject to the relation given by  $m(s,r)=d$. Let $H=\{r\}$ and $\mathfrak M^H$ be the free $\mathbb Z[q^{\frac {1}{2}},q^{-\frac {1}{2}}]$-module  with $\{m_w : w\in W^H_{s,r;d}\}$ as a basis. 

Let $x\in \{q,-1\}$. We set $L_s$ and $\Gamma_s$ to be the unique endomorphisms of $\mathfrak M^H$ satisfying
$$ L_s(m_w)= \left\{ \begin{array}{ll}
q\,m_{sw}+(q-1)\,m_w, & \mbox{if $sw  \lhd w$,} \\
m_{sw}, & \mbox{if $w\lhd sw \in W^H_{s,r;d}$,} \\
x\,m_w, & \mbox{if $w\lhd sw \notin W^H_{s,r;d}$,} 
\end{array} \right. $$
and 
$$ \Gamma_s(m_w)= \left\{ \begin{array}{ll}
m_{sw}, & \mbox{if $sw  \lhd w$,} \\
q\,m_{sw}+(q-1)\,m_w, & \mbox{if $w\lhd sw \in W^H_{s,r;d}$,} \\
x\,m_w, & \mbox{if $w\lhd sw \notin W^H_{s,r;d}$,} 
\end{array} \right. $$
for every basis element $m_w$. The endomorphisms $L_r$ and $\Gamma_r$ are defined analogously. 

For $p\in[0,d-1]$, we denote by $w_p$ the unique element in $W^H_{s,r;d}$ of length $p$: $w_p= \underbrace{\ldots s r s }_{p \text{ letters}}$. Thus $\{m_w : w\in W^H_{s,r;d}\}=\{m_{w_p} : p\in[0,d-1]\}$, and  
\begin{itemize}
\item $ L_s(m_{w_p})= \left\{ \begin{array}{ll}
q\,m_{w_{p-1}}+(q-1)\,m_{w_p}, & \mbox{if $p\in [1,d-1], \;p \text{ odd}$,} \\
m_{w_{p+1}}, & \mbox{if $p\in [0,d-1), \;p \text{ even}$,} \\
x\,m_{w_p}, & \mbox{if $p=d-1,  \;p \text{ even}$,} 
\end{array} \right. $
\item $ L_r(m_{w_p})= \left\{ \begin{array}{ll}
q\,m_{w_{p-1}}+(q-1)\,m_{w_p}, & \mbox{if $p\in (0,d-1], \;p \text{ even}$,} \\
m_{w_{p+1}}, & \mbox{if $p\in [1,d-1), \;p \text{ odd}$,} \\
x\,m_{w_p}, & \mbox{if either $p=0$, or $p=d-1 \text{ with $p$ odd}$,} 
\end{array} \right. $
\item $ \Gamma_s(m_{w_p})= \left\{ \begin{array}{ll}
m_{w_{p-1}}, & \mbox{if $p\in [1,d-1], \;p \text{ odd}$,} \\
q\,m_{w_{p+1}}+(q-1)\,m_{w_p}, & \mbox{if $p\in [0,d-1), \;p \text{ even}$,} \\
x\,m_{w_p}, & \mbox{if $p=d-1,  \;p \text{ even}$,} 
\end{array} \right. $
\item $ \Gamma_r(m_{w_p})= \left\{ \begin{array}{ll}
m_{w_{p-1}}, & \mbox{if $p\in (0,d-1], \;p \text{ even}$,} \\
q\,m_{w_{p+1}}+(q-1)\,m_{w_p}, & \mbox{if $p\in [1,d-1), \;p \text{ odd}$,} \\
x\,m_{w_p}, & \mbox{if either $p=0$, or $p=d-1 \text{ with $p$ odd}$.} 
\end{array} \right. $
\end{itemize}

Furthermore, we define an involutive automorphism $I$ of $\mathfrak M^H$ by letting
$I(m_{w_p})=m_{w_{d-p-1}}$, and we set
$$ \bar{s}= \left\{ \begin{array}{ll}
s, & \mbox{if $d$ even,} \\
r, &  \mbox{if $d$ odd,} 
\end{array} \right. 
\hspace{2cm}
\bar{r}= \left\{ \begin{array}{ll}
r, & \mbox{if $d$ even,} \\
s, & \mbox{if $d$ odd.} 
\end{array} \right.  $$
Note that $D_L(w_{d-p-1})=\{\bar{s}\}$, for all even $p\in [0,d-1)$.

\begin{thm}
\label{diagrammicommutano}
The following two diagrams are commutative:
\begin{center}
\begin{tikzcd}
\mathfrak M^H  \arrow[r, "I"] \arrow[d,  "L_s"] & \mathfrak M^H \arrow[d, "\Gamma_{\bar{s}}" ] \\
\mathfrak M^H  \arrow[r,  "I" ] & \mathfrak M^H
\end{tikzcd}
\hspace{2cm}
\begin{tikzcd}
\mathfrak M^H  \arrow[r, "I"] \arrow[d,  "L_r"] & \mathfrak M^H \arrow[d, "\Gamma_{\bar{r}}" ] \\
\mathfrak M^H  \arrow[r,  "I" ] & \mathfrak M^H
\end{tikzcd}
\end{center}
Moreover, $\mathfrak M^H$ is an $\mathcal H(W_{s,r;d})$-module with $T_s$ and $T_r$ acting as $\Gamma_s$ and $\Gamma_r$.
\end{thm}
\begin{proof}
Let us show that  $I \circ L_s (m_{w_p})=  \Gamma_{\bar s} \circ I(m_{w_p})$, for all $p\in [1,d-1]$. We have
 
 $I \circ L_s (m_{w_p})=  \left\{ \begin{array}{ll}
q\,m_{w_{d-p}}+(q-1)\,m_{w_{d-p-1}}, & \mbox{if $p\in [1,d-1], \;p \text{ odd}$,} \\
m_{w_{d-p-2}}, & \mbox{if $p\in [0,d-1), \;p \text{ even}$,} \\
x\,m_{w_{d-p-1}}, & \mbox{if $p=d-1,  \;p \text{ even}$.} 
\end{array} \right.$

If $d$ is even, then $\bar{s}=s$ and $d-p-1\not \equiv p \mod 2$. Thus
$$\Gamma_{\bar s} \circ I(m_{w_p})=\Gamma_{s} (m_{w_{d-p-1}})=  \left\{ \begin{array}{ll}
m_{w_{d-p-2}}, & \mbox{if $d-p-1\in [1,d-1], \;d-p-1 \text{ odd}$,} \\
q\,m_{w_{d-p}}+(q-1)\,m_{w_{d-p-1}}, & \mbox{if $d-p-1\in [0,d-1), \;d-p-1 \text{ even}$,} \\
x\,m_{w_{d-p-1}}, & \mbox{if $d-p-1=d-1,  \;d-p-1 \text{ even}$,} 
\end{array} \right.$$
namely
$$\Gamma_{\bar s} \circ I(m_{w_p})=  \left\{ \begin{array}{ll}
m_{w_{d-p-2}}, & \mbox{if $p\in [0,d-1), \;p \text{ even}$,} \\
q\,m_{w_{d-p}}+(q-1)\,m_{w_{d-p-1}}, & \mbox{if $p\in [1,d-1], \;p \text{ odd}$,} 
\end{array} \right.$$
since $d-1$ is odd.

If $d$ is odd, then $\bar{s}=r$ and $d-p-1 \equiv p \mod 2$. Thus
$$\Gamma_{\bar s} \circ I(m_{w_p})=\Gamma_{r} (m_{w_{d-p-1}})=  \left\{ \begin{array}{ll}
m_{w_{d-p-2}}, & \mbox{if $d-p-1\in (0,d-1], \;d-p-1 \text{ even}$,} \\
q\,m_{w_{d-p}}+(q-1)\,m_{w_{d-p-1}}, &  \mbox{if $d-p-1\in [1,d-1), \;d-p-1 \text{ odd}$,} \\
x\,m_{w_{d-p-1}}, & \mbox{if $d-p-1=0$ or $d-p-1=d-1 \text{ odd}$,} 
\end{array} \right.$$
namely
$$\Gamma_{\bar s} \circ I(m_{w_p})=  \left\{ \begin{array}{ll}
m_{w_{d-p-2}}, & \mbox{if $p\in [0,d-1), \;p \text{ even}$,} \\
q\,m_{w_{d-p}}+(q-1)\,m_{w_{d-p-1}}, &  \mbox{if $p\in [1,d-1], \;p \text{ odd}$,} \\
x\,m_{w_{d-p-1}}, & \mbox{if $p=d-1$,} 
\end{array} \right.$$
since $d-1$ is even.
 Therefore, the first diagram is commutative.
 
The argument to show that also the second diagram is commutative is entirely similar and left to the reader.

The last statement follows by the commutativity of the above diagrams, since $L_s$ and $L_r$ define a Hecke algebra action by Corollary~2.3 of \cite{Deo87} and $I$ is an isomorphism (note that there is a misprint in the definition of $L(s)$ at page 485 of \cite{Deo87}).
\end{proof}

Recall that, given a pircon $P$ and $w\in P$, we denote by $SPM_w$ the set of all special partial matchings of $w$. Notice that, if $M,N \in SPM_w$, then the orbit $\langle M,N \rangle (w)$ is dihedral by the definition of a special partial matching and Lemma~\ref{noorbita}. The following definition is a generalization of \cite[Definition~3.1]{BCM2}.

\begin{defi}
Let $P$ be a pircon,  $w \in P$, and  $M,N \in SPM_w$. We say that $M$ and $N$ are
\emph{strictly coherent} provided that  
\begin{itemize}
\item 
the rank of $\mathcal O$ divides the rank of $\langle M,N \rangle (w)$, if $\mathcal O$  is a dihedral interval,
\item 
the rank of $\mathcal O$ plus 1 divides the rank of $\langle M,N \rangle (w)$, if $\mathcal O$  is a chain,
\end{itemize}
for every orbit $\mathcal O$  of $\langle M,N \rangle $.
Moreover, if $S$ is a set of special partial matchings, then we say that $M$ and $N$ are \emph{$S$-coherent} provided that  $S$ has  a sequence $M_0,M_1, . . . , M_k$ of special partial
matchings of $w$ such that $M_0 = M$, $M_k = N$, and $M_i$ and $M_{i+1}$ are strictly coherent for all $i = 0,1, . . . , k-1$. For short, we write coherent instead of  $SPM_w$-coherent.
\end{defi}

Given a pircon  $P$, $w\in P$, and $M,N \in SPM_w$,  we denote by \( (W_{M,N},\{M,N\}) \)
the Coxeter system whose Coxeter generators are $M$ and $N$ subject to the relation given by $ m(M,N)=  \text{ rank }\langle M,N \rangle (w)= |\langle M,N \rangle (w)|/2$. 
For every orbit $\mathcal O$ of $\langle M,N \rangle$, we denote by 
 \( \mathfrak{M}^{\mathcal O} \) the free ${\mathbb Z}[q^{\frac{1}{2}},q^{-\frac{1}{2}}]$-module  with $\{m_u : u\in \mathcal O\}$ as a basis: $\mathfrak{M}^{\mathcal O}= \bigoplus _{u \in \mathcal O }{\mathbb Z}[q^{\frac {1}{2}},q^{-\frac {1}{2}}]m_u.$
\begin{thm}
\label{film}
Given a pircon $P$ and $w\in P$, let $M,N\in SPM_w$ and $\mathcal O$ be an orbit of $\langle M,N \rangle$. If $M$ and $N$ are strictly coherent, then $\mathfrak{M}^{\mathcal O}$ is a module over the Hecke algebra  \( {\mathcal H}(W_{M,N}) \) with 
\[T_M \cdot m_u= \left\{ \begin{array}{ll}
m_{M(u)}, & \mbox{if $M(u)  \lhd u$,} \\
q\,m_{M(u)}+(q-1)\,m_u, & \mbox{if $M(u) \rhd u$,} \\
x\,m_u, & \mbox{if $M(u) = u$,} 
\end{array} \right. \] 
and 
\[T_N \cdot m_u=  \left\{ \begin{array}{ll}
m_{N(u)}, & \mbox{if $N(u)  \lhd u$,} \\
q\,m_{N(u)}+(q-1)\,m_u, & \mbox{if $N(u) \rhd u$,} \\
x\,m_u, & \mbox{if $N(u) = u$,} 
\end{array} \right. \]
 for all $u\in  \mathcal O$.
\end{thm}
\begin{proof}
By Lemma~\ref{noorbita}, the orbit $\mathcal O$ can be either dihedral or chain-like.
If $\mathcal O$ is dihedral then we can proceed in a similar way as in \cite[Section 3]{BCM2}. So we assume that $\mathcal O$ is chain-like; let $d-1$ be its rank and suppose (without loss of generality) that $N(\hat{0}_{\mathcal O})= \hat{0}_{\mathcal O}$.

Recall that we denote by $W_{s,r;d}$ the dihedral Coxeter group with Coxeter generators $s$ and $r$ subject to the relation given by  $m(s,r)=d$. Let $H=\{r\}$ and $\mathfrak M^H$ be the free $\mathbb Z[q^{\frac {1}{2}},q^{-\frac {1}{2}}]$-module  with $\{m_w : w\in W^H_{s,r;d}\}$ as a basis. The two modules \( \mathfrak{M}^{\mathcal O} \) and  \( \mathfrak{M}^{H} \) are isomorphic and we identify them through the isomorphism sending $m_{\underbrace{\ldots M N M }_{p \text{ letters}}(\hat{0}_{\mathcal O})}$ to $m_{w_p}$.
Since $M$ and $N$ are strictly coherent, $d$ divides $m(M,N)$; so we have a morphism of Hecke algebras:
$$\begin{array}{ccc}
 {\mathcal H}(W_{M,N}) & \longrightarrow &  {\mathcal H}(W_{s,r;d}) \\
T_M & \mapsto & T_s\\
T_N & \mapsto & T_r
\end{array}$$
The pullback of the representation of Theorem~\ref{diagrammicommutano} is the desired representation.
\end{proof}

\section{Kazhdan--Lusztig $R^x$-polynomials for pircons}
\label{KLMMpoly}
In this section, we introduce and study the Kazhdan--Lusztig $R^x$-polynomials of a pircon $P$. The construction mimics that of \cite[Section~3]{BCM2}, but the proofs in this general setting are more complicated and use the results in Sections~\ref{orbits} and \ref{quattro}. Furthermore, for each $x\in\{q,-1\}$, there are possibly many different families of  Kazhdan--Lusztig $R^x$-polynomials attached to a pircon $P$. 

By the definition of a pircon,  we can fix a special partial matching of $v$,  for each $ v\in P\setminus \{\hat{0}_P\} $. If  $\mathcal M$ is the set of such fixed special partial matchings, then we call the pair $(P,\mathcal M)$ a \emph{refined pircon} and $\mathcal M$ a \emph{refinement} of $P$.

\begin{defi}
\label{klmm}
Let $x\in \{q,-1\}$.
Let $(P,\mathcal M)$ be a refined pircon, where $\mathcal M=  \{M_v \in SPM_v: v\in P\setminus \{\hat{0}\} \}$. The family of \emph{Kazhdan--Lusztig $R^{x}$-polynomials} $\{R^{x}_{u,w}(q)\}_{u,w\in P}\subseteq \mathbb Z[q]$ of $(P,\mathcal M)$ (or $R^{x}$-polynomials for short) is the unique family of polynomials satisfying  the following recursive property and initial conditions:
\begin{equation}
\label{klperpirconi}
 R_{u,w}^{x} (q)= \left\{ \begin{array}{ll}
R_{M_w(u),M_w(w)}^{x}(q), & \mbox{if $M_w(u)  \lhd u$,} \\
(q-1)R_{u,M_w(w)}^{x}(q)+qR_{M_w(u),M_w(w)}^{x}(q), & \mbox{if $M_w(u) \rhd u$,} \\
(q-1-x)R_{u,M_w(w)}^{x}(q), & \mbox{if $M_w(u) = u$,} 
\end{array} \right. 
\end{equation}
and 
$R_{w,w}^{x} (q)=1$ for all $w\in P$.
\end{defi}

Notice that 
\begin{equation}
\label{q-1-}
q-1-x=  \left\{ \begin{array}{rl}
-1, & \mbox{if $x=q$,} \\
q, & \mbox{if $x=-1$,} 
\end{array} \right. 
\end{equation} 
so  the $R^{q}$-polynomials are the $R$-polynomials in whose recursion $-1$ is appearing whereas the $R^{-1}$-polynomials are the $R$-polynomials in whose recursion $q$ is appearing. The (at first glance unnatural) choice follows the usual terminology for the parabolic Kazhdan--Lusztig $R$-polynomials (see Subsection~\ref{para}). The two families of parabolic Kazhdan--Lusztig $R^q$-polynomials  and $R^{-1}$-polynomials are associated with two modules, usually denoted $M^q$ and $M^{-1}$, respectively. For these two modules, the adopted notation is more natural than the opposite one. 

The two families of $R$-polynomials satisfy the following properties.
\begin{pro}
\label{relazione}
Let $(P,\mathcal M)$ be a refined pircon with rank function $\rho$. If $u,w\in P$, then 
\begin{enumerate}
\item $\deg R^{-1}_{u,w}(q)=\rho(w)-\rho(u)$,
\item $R^{q}_{u,w}(0)=(-1)^{\rho(w)-\rho(u)}$,
\item $R^{-1}_{u,w}(q)=(-q)^{\rho(w)-\rho(u)} \; R^{q}_{u,w}(q^{-1}).$
\end{enumerate}
\end{pro}
\begin{proof}
The first two statements are straighforward by Definition~\ref{klmm}. For the third, we need to show that the polynomials $(-q)^{\rho(y)-\rho(x)} \; R^{q}_{x,y}(q^{-1})$, for $x,y\in P$, satisfy the recursive property and the initial conditions in Definition~\ref{klmm}. This easy computation is left to the reader.
\end{proof}

Let $(P,\mathcal M )$ be a refined pircon and $w\in P$. Mimicking \cite[Subsection~4.1]{Mtrans}, we say that a special partial matching $M$ of $w$ \emph{calculates} the Kazhdan--Lusztig $R^{x}$-polynomials of $(P,\mathcal M)$ (or  is \emph{calculating}, for short) provided that, for all $u \in P$, $u \leq w$, the following holds:
\begin{equation*}
 R_{u,w}^{x} (q)= \left\{ \begin{array}{ll}
R_{M(u),M(w)}^{x}(q), & \mbox{if $M(u)  \lhd u$,} \\
(q-1)R_{u,M(w)}^{x}(q)+qR_{M(u),M(w)}^{x}(q), & \mbox{if $M(u) \rhd u$,} \\
(q-1-x)R_{u,M(w)}^{x}(q), & \mbox{if $M(u) = u$.} 
\end{array} \right. 
\end{equation*}
Thus the matchings of $\mathcal M$ are calculating by definition. 

\begin{rem}
In general, the Kazhdan--Lusztig $R^{x}$-polynomials of a refined pircon $(P,\mathcal M)$  depend on the refinement $\mathcal M$. For example, let $P$ be the pircon in Figure~\ref{nondircone}, and let $\mathcal M$ and $\mathcal M'$ be any two refinements of  $P$  such that  $\mathcal M$ contains the dashed special partial matching while  $\mathcal M'$ contains the solid special partial matching.
\begin{figure}[h]
\begin{center}
$$
\begin{tikzpicture} 
 \draw[fill=black]{(4,1) circle(3pt)};
 \node[left] at (4,1){$v$};
 \draw[fill=black]{(6,1) circle(3pt)};
\draw[fill=black]{(5,2) circle(3pt)};
\draw[fill=black]{(5,3) circle(3pt)};
\draw[fill=black]{(5,4) circle(3pt)};
\node[above] at (5,4){$\hat{1}_P$};
\draw[fill=black]{(5,0) circle(3pt)};

 \draw[line width=3pt, line width=3pt]{(5,0)--(4,1)};
  \draw[line width=3pt, line width=3pt]{(5,2)--(6,1)}; 
 \draw{(5,2)--(5,3)}; 
\draw[dashed, line width=3pt]{(5,2)--(4,1)};
\draw{(5,2)--(4,1)};
\draw[dashed, line width=3pt]{(5,0)--(6,1)};
\draw{(5,0)--(6,1)};
\draw[dashed, line width=3pt]{(4.9,3)--(4.9,4)};
 \draw[line width=3pt, line width=3pt]{(5.1,3)--(5.1,4)}; 
\end{tikzpicture}$$
\end{center}
\caption{\label{nondircone} A pircon and two of its special partial matchings}
\end{figure}
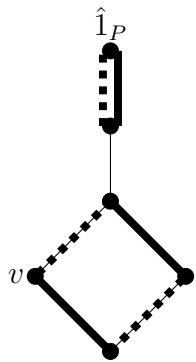
No matter what the other special partial matchings in $\mathcal M$ and $\mathcal M'$ are, the Kazhdan--Lusztig $R^{x}$-polynomial $R^{x}_{v,\hat{1}_P}$ of $(P,\mathcal M)$ is $(q-1)(q-1-x)^2$ and the Kazhdan--Lusztig $R^{x}$-polynomial $R^{x}_{v,\hat{1}_P}$ of $(P,\mathcal M')$ is $(q-1)^2 (q-1-x)$. Therefore, the dashed matching does not calculate the Kazhdan--Lusztig $R^{x}$-polynomials of $(P,\mathcal M)$ and the solid matching does not calculate the Kazhdan--Lusztig $R^{x}$-polynomials of $(P,\mathcal M')$. Notice that the dashed matching and the solid matching are not coherent.
\end{rem}

\begin{defi}
Let $(P,\mathcal M )$ be a refined pircon and $w\in P$. We say that a special partial matching $M$ of $w$  is \emph{strongly calculating} provided that the restriction of $M$ to $P_{\leq z}$ is calculating for all $z\in P$ such that $z\leq w$ and $M(z)\lhd z$.
\end{defi}
Notice that, by Lemma~\ref{restringe}, the restriction of $M$ to $P_{\leq z}$ is indeed a special partial matching for all $z\in P$ such that $M(z)\lhd z$.
\begin{thm}
\label{secommutano}
Let $(P, \mathcal M)$ be a refined pircon, $w \in P$, and $M$ be a special partial matching  of $w$. Suppose that
\begin{itemize}
\item the restriction of $M$ to $P_{\leq z}$ is calculating, for all $z\in P$ such that $z< w$ and $M(z)\lhd z$, and
\item there exists a strongly calculating special partial matching $N$ of $w$ 
that is strictly coherent with $M$. 
\end{itemize}
Then $M$  is strongly calculating.
\end{thm}
\begin{proof}
By hypothesis, the restriction of $M$ to $P_{\leq z}$ is calculating, for all $z\in P$ such that $z< w$ and $M(z)\lhd z$. Hence, we only need to show that $M$ is calculating: if we fix $u\in P$, with $u\leq w$, then  we need to prove 
\begin{equation}
\label{44}
 R_{u,w}^{x} (q)= \left\{ \begin{array}{ll}
R_{M(u),M(w)}^{x}(q), & \mbox{if $M(u)  \lhd u$,} \\
(q-1)R_{u,M(w)}^{x}(q)+qR_{M(u),M(w)}^{x}(q), & \mbox{if $M(u) \rhd u$,} \\
(q-1-x)R_{u,M(w)}^{x}(q), & \mbox{if $M(u) = u$.} 
\end{array} \right. 
\end{equation}

Let $m$ be the rank of $ \langle M,N \rangle (w)$ and $\mathcal O$ be the orbit $\langle M,N \rangle (u)$. Recall from Section~\ref{quattro} that  \( (W_{M,N},\{M,N\}) \) denotes the Coxeter system whose Coxeter generators are $M$ and $N$ subject to the relation given by \( m(M,N)= m$, and that  $\mathfrak{M}^{\mathcal O}$ denotes the free ${\mathbb Z}[q^{\frac{1}{2}},q^{-\frac{1}{2}}]$-module  with $\{m_v : v\in \mathcal O\}$ as a basis.

By Theorem \ref{film}, $\mathfrak{M}^{\mathcal O}$ is a  \( \mathcal H(W_{M,N}) \)-module with $T_M$ and $T_N$ acting as follows:
\[T_M \cdot m_v= \left\{ \begin{array}{ll}
m_{M(v)}, & \mbox{if $M(v)  \lhd v$,} \\
qm_{M(v)}+(q-1)m_v, & \mbox{if $M(v) \rhd v$,} \\
(q-1-x)m_v, & \mbox{if $M(v) = v$.} 
\end{array} \right. \] 
and 
\[T_N \cdot m_v=  \left\{ \begin{array}{ll}
m_{N(v)}, & \mbox{if $N(v)  \lhd v$,} \\
qm_{N(v)}+(q-1)m_v, & \mbox{if $N(v) \rhd v$,} \\
(q-1-x)m_v, & \mbox{if $N(v) = v$,} 
\end{array} \right. \]
 for all $v\in  \mathcal O$ (recall \eqref{q-1-}).

Given $f\in \mathfrak{M}^{\mathcal O}$ and $z\in P$, we denote by $f^z$   the polynomial
\( \sum_{v\in \mathcal O} f_v(q)R^x_{v,z}(q) \in \mathbb Z[q]\), provided that $f=\sum_{v\in \mathcal O} f_v(q)\,m_v$. 
In this notation, (\ref{44}) can be reformulated as
$R_{u,w}^{x} (q)= (T_M \cdot m_u)^{M(w)}$. 
Since $N$ is calculating, 
$R_{u,w}^{x} (q)= (T_N \cdot m_u)^{N(w)}$, and hence we are done if we prove 
\begin{equation*}
(T_M \cdot m_u)^{M(w)}=(T_N \cdot m_u)^{N(w)}.
\end{equation*}
For all $z<w$ such that $M(z)\lhd z$, the matching $M$ restricts to a special partial matching of $z$,  which is calculating  (by hypothesis). If $f=\sum_{v\in \mathcal O} f_v(q)\,m_v \in \mathfrak{M}^{\mathcal O}$, then
\begin{eqnarray*}
f^z &=& \sum_{v\in \mathcal O} f_v(q)R^x_{v,z}(q)\\
&=&\sum_{\{v:M(v)\lhd v\}}f_v(q)R^x_{M(v),M(z)}(q)+\sum_{\{v: v\lhd M(v)\}}f_v(q)\Big(qR^x_{M(v),M(z)}(q)+(q-1)R^x_{v,M(z)}(q) \Big) \\
&+&\sum_{\{v: v=M(v)\}}f_v(q) (q-1-x)R^x_{v,M(z)}(q)
\end{eqnarray*}
and
\begin{eqnarray*}
(T_M \cdot f)^{M(z)}&=&\Big( \sum_{v\in \mathcal O} f_v(q) T_M \cdot m_v \Big)^{M(z)}\\
&=&\Big(\sum_{\{v:M(v)\lhd v\}}  f_v(q) \, m_{M(v)}+ \sum_{\{v: v\lhd M(v)\}}f_v(q) \Big(q\,{m_{M(v)}}+(q-1)\,m_v \Big) \\
&+ &  \sum_{\{v: v= M(v)\}}f_v(q) (q-1-x)\,m_v  \Big)^{M(z)}\\
&=&\sum_{\{v:M(v)\lhd v\}}f_v(q)R^x_{M(v),M(z)}(q)+\sum_{\{v: v\lhd M(v)\}}f_v(q)\Big(qR^x_{M(v),M(z)}(q)+(q-1)R^x_{v,M(z)}(q) \Big) \\
&+&\sum_{\{v: v= M(v)\}}f_v(q) (q-1-x)R^x_{v,M(z)}(q).
\end{eqnarray*}

Hence
$$f^{z}=(T_M \cdot f)^{M(z)},$$
for all $f\in \mathfrak{M}^{\mathcal O}$ and all $z<w$ such that $M(z)\lhd z$.
Since $N$ is strongly calculating, an analogous computation yields
$$f^{z}=(T_N \cdot f)^{N(z)},$$
for all $f\in \mathfrak{M}^{\mathcal O}$ and all $z<w$ such that $N(z)\lhd z$. An alternated use of these two formulae implies
$$ (T_M \cdot m_u)^{M(w)}=(T_NT_M \cdot m_u)^{NM(w)} 
= \cdots= (\underbrace{\cdots T_MT_NT_M}_{m} \cdot m_u)^{\scriptsize {\underbrace{\cdots MNM}_{m}}(w)},
$$
and 
\[
(T_N \cdot m_u)^{N(w)}=(T_MT_N \cdot m_u)^{MN(w)} 
= \cdots=(\underbrace{\cdots T_NT_MT_N}_{m} \cdot m_u)^{\scriptsize {\underbrace{\cdots NMN}_{m}}(w)}.
\]
Since $\underbrace{\cdots T_MT_NT_M}_{m}= \underbrace{\cdots T_NT_MT_N}_{m}$ and \( \underbrace{\cdots MNM}_{m}(w)=\underbrace{\cdots NMN}_{m}(w) \), the desired equality
\begin{equation*}
(T_M \cdot m_u)^{M(w)}=(T_N \cdot m_u)^{N(w)}.
\end{equation*}
holds.
\end{proof}  

\begin{defi}
\label{sistema}
We say that $(P,S)$ is a \emph{pircon system} provided that
\begin{enumerate}
\item $P$ is a pircon,
\item $S\subseteq \bigcup_{w\in P\setminus \{\hat{0}_P\}}SPM_w$
\item
\label{tre}  for all $w\in P\setminus \{\hat{0}_P\}$, there exists $M\in S$ such that $M(w)$ is defined and $M(w)\lhd w$,
\item for all $w\in P\setminus \{\hat{0}_P\}$ and all $M,N\in S$ such that $M(w)$ and $N(w)$ are defined and satisfy $M(w)\lhd w$ and $N(w)\lhd w$, the restrictions of $M$ and $N$ to $P_{\leq w}$ are $S$-coherent.
\end{enumerate}
\end{defi}

Theorem~\ref{secommutano} implies the following result.
\begin{cor}
\label{unici}
Let $(P,S)$ be a pircon system and $x\in\{q,-1\}$. All refinements $\mathcal M$ of $P$, with $\mathcal M \subseteq S$, yields the same family of Kazhdan--Lusztig $R^x$-polynomials (for which, all matchings in $S$ are strongly calculating).
\end{cor}
\begin{proof}
First note that, if $(P,S)$ is a pircon system, then also $(P,\overline{S})$ is a pircon system, where $\overline{S}$ is the set of special partial matchings obtained from $S$ by adding all restrictions of $M$ to $P_{\leq z}$, for all $M\in S$ and $z\in P$ such that $z\leq w$ and $M(z)\lhd z$. Thus, we may suppose that $S$ is closed under taking such restrictions, i.e. $S=\overline{S}$.

Choose  an arbitrary refinement $\mathcal M$ of $P$, with $\mathcal M \subseteq S$ (whose existence is assured  by \eqref{tre} of Definition~\ref{sistema}). We need to show that all special partial matchings of $w$ belonging to $S$ are calculating, for all $w\in P$. We use induction on $\rho(w)$, the case  $\rho(w)=1$ being clear (as before, $\rho$ denotes the rank function of $P$).

Suppose $\rho(w)>1$. Let $M\in S$ be a special partial matching of $w$ and denote by $N$ the unique special partial matching of $w$ belonging to $\mathcal M$. If $M=N$, then the assertion is clear since $N$ is calculating by definition. Otherwise, $M$ and $N$ are $S$-coherent.

Suppose first that $M$ and $N$ are strictly coherent. By the induction hypothesis and the fact that  $S=\overline{S}$ holds, the restriction of $M$ to $P_{\leq z}$ is calculating, for all $z\in P$ such that $z< w$ and $M(z)\lhd z$, and, moreover, $N$ is strongly calculating. Thus, we can apply Theorem~\ref{secommutano}, which implies that $M$ is calculating. 

If $M$ and $N$ are not strictly coherent, we can conclude by transitivity, since $M$ and $N$ are $S$-coherent.
\end{proof}

\begin{defi}
\label{dircone}
A pircon $D$ is a \emph{dircon} provided that any two special partial matchings $M,N \in SPM_w$ are coherent,  for all $w \in D$. 
\end{defi}
In other words, a pircon $D$ is a dircon if and only if $(D, \bigcup_{w\in P\setminus \{\hat{0}_P\}}SPM_w)$ is a  pircon system. By Corollary~\ref{unici}, for both $x=q$ and $x=-1$, a dircon has a unique family of Kazhdan--Lusztig $R^x$-polynomials.

The terminology comes from the fact that dircons relate to pircons in the same way as diamonds relate to zircons (see \cite[Definition~3.2]{BCM2}).

\section{Applications}
In this section, we show that examples of Kazhdan--Lusztig $R^x$-polynomials of pircons include Kazhdan--Lusztig--Vogan $R$ and $Q$-polynomials for fixed point free involutions and parabolic Kazhdan--Lusztig polynomials of Coxeter groups.

\subsection{Kazhdan--Lusztig--Vogan $R$-polynomials for fixed point free involutions}
The orbits of the action of  $Sp(2n,C)$  on the flag variety of $SL(2n,C)$ are parametrized by the set of fixed point free involutions in the symmetric group $S_{2n}$, or, equivalently, by the set 
$\iota= \{\theta (w^{-1} ) w : w\in S_{2n}\}$ of \emph{twisted identities}, where $\theta$ is the involutive automorphism of $S_{2n}$ sending $s_i$ to $s_{2n-i}$, for all $i\in [1,n]$ (here $s_k$ denotes the simple transposition $(k,k+1)$, for all $k\in[1,2n-1]$). Note that $\theta (w )=w_0 w w_0 $ for all $w\in S_{2n}$, where $w_0$ is the longest permutation, i.e. the reverse permutation sending $i$ to $2n-i+1$, for all $i\in[1,2n]$. The set $\iota$ of twisted identities is endowed with a poset structure induced by the Bruhat order (all covering relations in this subsection refer to the induced poset structure).  In this setting, the associated Kazhdan--Lusztig--Vogan $R$-polynomials are indexed by pairs of elements in $\iota$ and are uniquely determined by the following recursive formula and initial values (see \cite[Proposition 5.1]{H}).
\begin{pro}
\label{klv}
The family of  Kazhdan--Lusztig--Vogan $R$-polynomials $\{R_{x,y}\}_{x,y\in \iota}$ for fixed point free involutions satisfies the following properties:
\begin{itemize}
\item $R_{e,e}=1$,
\item $R_{x,y}=0$ if $x\not\leq y$,
\item if $s\in D_R(w)$ then 
\begin{equation}
\label{klv2}
 R_{u,w} (q)= \left\{ \begin{array}{ll}
R_{u\ast s,w\ast s}(q), & \mbox{if $u\ast s  \lhd u$,} \\
(q-1)R_{u,w\ast s}(q)+qR_{u\ast s, w\ast s}(q), & \mbox{if $u\ast s \rhd u$,} \\
-R_{u,w\ast s}(q), & \mbox{if $u\ast s =u$,} 
\end{array} \right. 
\end{equation}
where $x\ast s= \theta (s) x s$ for all $x\in \iota$. 
\end{itemize}
\end{pro}
A similar result holds for the associated Kazhdan--Lusztig--Vogan $Q$-polynomials (see \cite[Proposition 5.3]{H}). We refer the reader to \cite{AH} for more details on this subject.

We want to show that the Kazhdan--Lusztig--Vogan $R$-polynomials and $Q$-polynomials for fixed point free involutions  lie in the theory of  Kazhdan--Lusztig $R^x$-polynomials of pircons. Indeed, we have the following result.
\begin{thm}
\label{iotapircone}
The poset $\iota$ is a dircon. Furthermore, the Kazhdan--Lusztig--Vogan $R$-polynomials and $Q$-polynomials for fixed point free involutions coincide, respectively, with the Kazhdan--Lusztig $R^q$-polynomials and $R^{-1}$-polynomials of $\iota$ as a dircon. 
 \end{thm} 
 In order to prove Theorem~\ref{iotapircone}, we need some preliminary  observations.
We refer to \cite{AH} for more details concerning the set of twisted identities and its special partial matchings. As proved in \cite[Theorem 4.3]{AH}, given $ w\in \iota$  and $s\in D_R(w)$, the map $x\mapsto x \ast s$ is a special partial matching of the lower interval $[e,w]$ of $\iota$. Following \cite{AH}, we call a special partial matching of  this form  a \emph{conjugation} special partial matching.
 \begin{lem}
 \label{7.3}
 Let $w\in \iota\setminus{e}$, and $M$ and $N$ be two distinct conjugation special partial matchings of $w$, say $M(x)=x \ast s_i$ and $M(x)=x \ast s_j$ for all $x\leq w$. If $|i-j|>1$ (respectively, $|i-j|=1$) then any orbit of the action of $\langle M,N \rangle$ is of one of the types in Figure~\ref{differenza>1} (respectively, Figure~\ref{differenza=1}), where $M$ and $N$ are colored in solid and dashed black.
 \end{lem}
\begin{proof}
If $|i-j|>1$ then  $m(s_i,s_j)=m(\theta(s_i),\theta(s_j))=2$; therefore $M$ and $N$ commute since $MN(x)=\theta(s_i)\theta(s_j)xs_js_i=\theta(s_j)\theta(s_i)xs_is_j=NM(x)$ for all $x\leq w$. If $|i-j|=1$ then  $m(s_i,s_j)=m(\theta(s_i),\theta(s_j))=3$; therefore $MNM=NMN$ since $$MNM(x)=\theta(s_i)\theta(s_j)\theta(s_i)xs_is_js_i=\theta(s_j)\theta(s_i)\theta(s_j)xs_js_is_j=NMN(x)$$ for all $x\leq w$. The result follows by Lemma~\ref{noorbita}.
\end{proof}

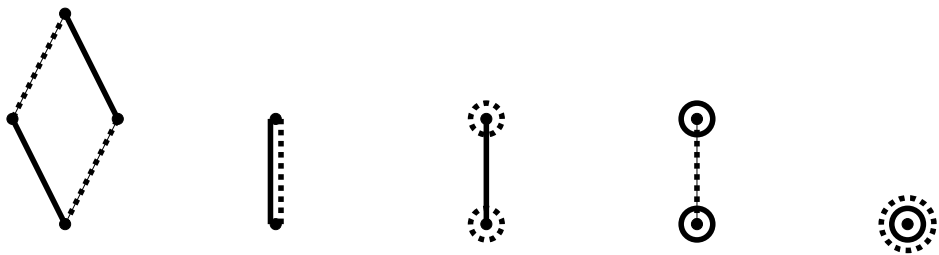
\begin{figure}[h]
\begin{center}
$$
\scalebox{.7}{
\begin{tikzpicture}
\draw[fill=black]{(0,0) circle(3pt)};
\draw[fill=black]{(1,2) circle(3pt)};
\draw[fill=black]{(0,4) circle(3pt)};
\draw[fill=black]{(-1,2) circle(3pt)};
\draw[dashed, line width=3pt]{(0,4)--(-1,2)};
 \draw[dashed, line width=3pt]{(0,0)--(1,2)}; 
\draw[line width=3pt]{(0,4)--(1,2)};
 \draw[line width=3pt]{(0,0)--(-1,2)}; 
 \draw{(-1,2)--(0,4)}; 
  \draw{(1,2)--(0,0)}; 

  \draw[fill=black]{(4,2) circle(3pt)};
\draw[fill=black]{(4,0) circle(3pt)};
 \draw[ line width=3pt]{(3.9,0)--(3.9,2)}; 
 \draw[dashed,  line width=3pt]{(4.1,0)--(4.1,2)}; 
 
 \draw[fill=black]{(8,2) circle(3pt)};
\draw[fill=black]{(8,0) circle(3pt)};
\draw[dashed, line width=3pt]{(8,0) circle (.3)};
\draw[dashed, line width=3pt]{(8,2) circle (.3)};
 \draw[ line width=3pt]{(8,0)--(8,2)}; 
\draw{(8,0)--(8,2)}; 

 \draw[fill=black]{(12,2) circle(3pt)};
\draw[fill=black]{(12,0) circle(3pt)};
\draw[ line width=3pt]{(12,0) circle (.3)};
\draw[ line width=3pt]{(12,2) circle (.3)};
 \draw[dashed, line width=3pt]{(12,0)--(12,2)}; 
\draw{(12,0)--(12,2)}; 

\draw[fill=black]{(16,0) circle(3pt)};
\draw[ line width=3pt]{(16,0) circle (.3)};
\draw[dashed, line width=3pt]{(16,0) circle (.5)};
\end{tikzpicture}}$$
\end{center}
\caption{\label{differenza>1} Orbits when $m(s_i,s_j)=2$}
\end{figure}
  
      \begin{figure}[h]
\begin{center}
$$
\scalebox{.7}{
\begin{tikzpicture}

\draw[fill=black]{(0,0) circle(3pt)};
\draw[fill=black]{(-1,2) circle(3pt)};
\draw[fill=black]{(-1,4) circle(3pt)};
\draw[fill=black]{(1,2) circle(3pt)};
\draw[fill=black]{(1,4) circle(3pt)};
\draw[fill=black]{(0,6) circle(3pt)};
\draw[dashed, line width=3pt]{(-1,4)--(-1,2)};
 \draw[dashed, line width=3pt]{(0,0)--(1,2)}; 
 \draw[dashed, line width=3pt]{(1,4)--(0,6)}; 
\draw[line width=3pt]{(1,4)--(1,2)};
 \draw[line width=3pt]{(0,0)--(-1,2)}; 
 \draw[line width=3pt]{(-1,4)--(0,6)}; 
 \draw{(-1,2)--(1,4)}; 
 \draw{(1,2)--(-1,4)}; 
  \draw{(-1,2)--(-1,4)}; 
  \draw{(1,2)--(0,0)}; 
  \draw{(1,4)--(0,6)}; 
 
  \draw[fill=black]{(4,2) circle(3pt)};
\draw[fill=black]{(4,0) circle(3pt)};
 \draw[ line width=3pt]{(3.9,0)--(3.9,2)}; 
 \draw[dashed,  line width=3pt]{(4.1,0)--(4.1,2)}; 
 
 \draw[fill=black]{(8,2) circle(3pt)};
\draw[fill=black]{(8,0) circle(3pt)};
\draw[fill=black]{(8,4) circle(3pt)};
\draw[dashed,  line width=3pt]{(8,4)--(8,2)};
\draw[dashed, line width=3pt]{(8,0) circle (.3)};
\draw[ line width=3pt]{(8,4) circle (.3)};
 \draw[ line width=3pt]{(8,0)--(8,2)}; 
\draw{(8,0)--(8,2)}; 

 \draw[fill=black]{(12,2) circle(3pt)};
\draw[fill=black]{(12,0) circle(3pt)};
\draw[fill=black]{(12,4) circle(3pt)};
\draw[  line width=3pt]{(12,4)--(12,2)};
\draw[ line width=3pt]{(12,0) circle (.3)};
\draw[dashed, line width=3pt]{(12,4) circle (.3)};
 \draw[dashed, line width=3pt]{(12,0)--(12,2)}; 
\draw{(12,0)--(12,2)}; 

\draw[fill=black]{(16,0) circle(3pt)};
\draw[ line width=3pt]{(16,0) circle (.3)};
\draw[dashed, line width=3pt]{(16,0) circle (.5)};
\end{tikzpicture}
}$$
\end{center}
\caption{\label{differenza=1} Orbits when $m(s_i,s_j)=3$}
\end{figure}
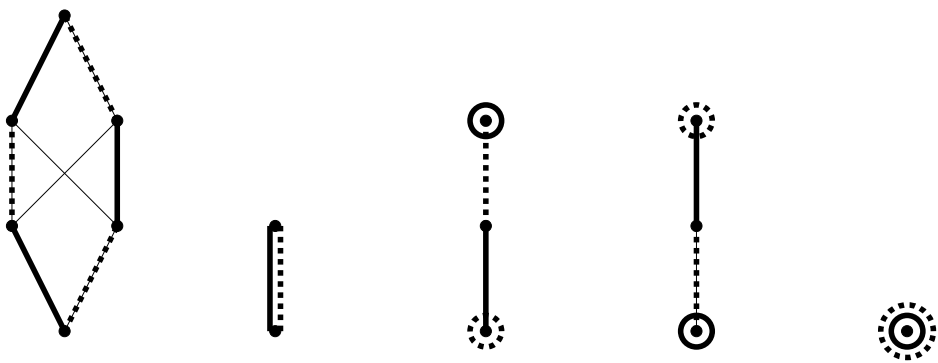
\begin{pro}
\label{strettamente}
Let $w\in \iota\setminus{e}$, and $M$ and $N$ be two distinct conjugation special partial matchings of $w$. Then $M$ and $N$ are strictly coherent if and only if $M(w)\neq N(w)$.
\end{pro}
\begin{proof}
Straightforward by Lemma~\ref{7.3}.
\end{proof}
\begin{proof}[Proof of Theorem~\ref{iotapircone}]
Let $w\in \iota\setminus{e}$ and $M,N \in SPM_w$: we need to show that $M$ and $N$ are coherent. Clearly, we may assume $\rho(w)>1$ ($\rho$ being the rank function of $\iota$). If $M$ is not a conjugation special partial matching, then there exists a conjugation special partial matching $M'$ of $w$ commuting with $M$ such that $M'(w)\neq M(w)$ (see Proposition 4.8 of \cite{AH}); hence $\langle M, M'\rangle (w)$ is a dihedral orbit of rank 2, while every other orbit is dihedral of rank 1 or 2 (indeed, it cannot be chain-like since $M$ and $M'$ have no fixed points, again by Proposition 4.8 of \cite{AH}). Therefore, $M$ and $M'$ are strictly coherent. 
So we may assume that $M$ and $N$ are both conjugation special partial matchings (say $M(x)=x \ast s$ and $M(x)=x \ast s'$, for all $x\in \iota$ such that $x\leq w$), and  that $M(w)=N(w)$ by Proposition~\ref{strettamente}.

Let $x=M(w)=N(w)$ and $r\in D_R(x)$. If $r$ commutes with one among $s$ and $s'$, then $w\ast r \lhd w$ by Lemma~\ref{7.3}. Hence $C_r: x\mapsto x \ast r$  is a conjugation special partial matching of $w$; since  $ C_r$ is  strictly coherent with both $M$ and $N$ by Proposition~\ref{strettamente}, $M$ and $N$ are coherent. Thus we are reduced to the case where $(s,r,s')=(s_{i-1},s_i, s_{i+1})$ and  $\{s_i\} =D_R(x)$ holds for a certain  $i\in [2,2n-2]$. But this is impossible since $x \ast s_{i-1}=x \ast s_{i+1}\rhd x$ implies
$$(x(i-1),x(i),x(i+1),x(i+2))= (2n-i-1, 2n-i,2n-i+1,2n-i+2),$$
and hence $s_i\notin D_R(x)$.

In order to show that the Kazhdan--Lusztig--Vogan $R$-polynomials and $Q$-polynomials coincide with the Kazhdan--Lusztig $R^q$-polynomials and $R^{-1}$-polynomials of $\iota$ as a dircon, compare  \eqref{klperpirconi} with \eqref{klv2} and \cite[Proposition 5.3]{H}.
\end{proof}

\subsection{Parabolic Kazhdan--Lusztig polynomials of a Coxeter system}
\label{para}
In this subsection, we fix an arbitrary  Coxeter system $(W,S)$ and a subset $H \subseteq S$.   The Bruhat order on $W$ induces an ordering on the set of minimal coset representatives  $W^{H}$ and, for all $u,v \in W^{H}$, on the parabolic  interval $[u,v]^{H}=\{z\in [u,v] : z\in W^H \}$. 

For the reader's convenience, we recall the following theorem (see  \cite[Sections~2 and 3]{Deo87} for a proof).
\begin{thm}
\label{7.1}
Let $x \in \{ -1,q \}$. 
 Then  there is a unique family of polynomials $\{ R_{u,v}^{H,x}(q)
\} _{u,v \in W^{H}} \subseteq {\bf Z}[q]$ such that, for all $u,v \in W^{H}$:
\begin{enumerate}
\item $R^{H,x}_{u,v}(q)=0$ if $u \not \leq v$;
\item $R^{H,x}_{u,u}(q)=1$;
\item if $u<v$ and $s \in D_{L}(v)$, then
\[ R_{u,v}^{H,x} (q)= \left\{ \begin{array}{ll}
R_{su,sv}^{H,x}(q), & \mbox{if $s \in D_{L}(u)$,} \\
(q-1)R_{u,sv}^{H,x}(q)+qR_{su,sv}^{H,x}(q), & \mbox{if $s \notin D_{L}(u)$
and $su \in W^{H}$,} \\
(q-1-x)R_{u,sv}^{H,x}(q), & \mbox{if $s \notin D_{L}(u)$ and $su \notin
W^{H}$.}
\end{array} \right. \]
\end{enumerate}
Moreover,  there is a unique family of polynomials $\{ P^{H,x}_{u,v}(q) \} 
_{u,v \in  W^{H}}  \subseteq {\bf Z} [q]$, such that, for all 
$u,v \in W^{H}$:
\begin{enumerate}
\item $P^{H,x}_{u,v}(q)=0$ if $u \not \leq v$;
\item $P^{H,x}_{u,u}(q)=1$;
\item deg$(P^{H,x}_{u,v}(q)) \leq  \frac{1}{2}\left(
\ell(v)-\ell(u)-1 \right) $, if $u < v$;
\item 
\( q^{\ell(v)-\ell(u)} \, P^{H,x}_{u,v} (q^{-1})
= \sum _{z\in[u,v]_H}   R^{H,x}_{u,z}(q) \,
P^{H,x}_{z,v}(q). \)
\end{enumerate}
\end{thm}

The polynomials $R^{H,x}_{u,v}(q)$ and $P^{H,x}_{u,v}(q)$ 
are  the {\em
parabolic Kazhdan--Lusztig $R$-polynomials} and {\em parabolic  Kazhdan--Lusztig polynomials} of $W^{H}$ of type $x$. 
If $H=\emptyset$ then $R_{u,v}^{\emptyset ,-1}(q)=R_{u,v}^{\emptyset ,q}(q)$ and $
P_{u,v}^{\emptyset ,-1}(q)=P_{u,v}^{\emptyset ,q}(q)$
 are the ordinary Kazhdan--Lusztig $R$-polynomials $R_{u,v}(q)$ and Kazhdan--Lusztig polynomials $P_{u,v}(q)$  of $W$.
 
 \begin{rem}
\label{senso}
The parabolic Kazhdan--Lusztig $R$-polynomials and the parabolic Kazhdan--Lusztig polynomials are equivalent. More precisely, given $w\in W^H$, one can compute the family $\{P_{u,v}^{H,x}(q)\}_{u,v\in [e,w]^H}$ once one knows   the family $\{R_{u,v}^{H,x}(q)\}_{u,v\in [e,w]^H}$, and vice versa.
\end{rem}

\bigskip

Let $w \in W^H$. 
As in \cite{Mtrans, M}, we say that a special matching of $[e,w]$ is {\em $H$-special} provided that 
$$u \leq w, \; u \in W^H, \; M(u) \lhd u \Rightarrow M(u) \in W^H.$$

Note that a $\emptyset$-special matching is a special matching and that a left multiplication matching  is $H$-special 
for all $H \subseteq S$. 

\bigskip
The following lemmas are needed in the proof of the main result of this subsection.
\begin{lem}
\label{noconfigurazione}
Let $(W,S)$ be a Coxeter system. Fix $H\subseteq S$. The Hasse diagram of $W$ cannot have the following configuration as a subdiagram:
$$
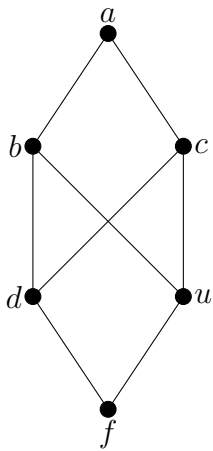
\begin{tikzpicture}

\node[below] at (0,-0.5){$f$};
\node[left] at (-1,1){$d$};
\node[right] at (1,1){$u$};
\draw{(0,-0.5)--(-1,1)}; 
\draw{(0,-0.5)--(1,1)}; 

\draw{(-1,3)--(-1,1)}; 
\draw{(1,3)--(1,1)}; 
\draw{(-1,3)--(1,1)}; 
\draw{(1,3)--(-1,1)}; 
\draw[fill=black]{(1,1) circle(3pt)};
\draw[fill=black]{(1,3) circle(3pt)};
\draw[fill=black]{(-1,3) circle(3pt)};
\draw[fill=black]{(-1,1) circle(3pt)};

\draw[fill=black]{(0,-0.5) circle(3pt)};

\node[left] at (-1,3){$b$};
\node[right] at (1,3){$c$};
\draw{(-1,3)--(0,4.5)}; 
\draw{(1,3)--(0,4.5)}; 
\draw[fill=black]{(0,4.5) circle(3pt)};
\node[above] at (0,4.5){$a$};
\end{tikzpicture}
$$
with $a,b,c \notin W^H$ and $u,f\in W^H$.
\end{lem}
\begin{proof}
Since $a\notin W^H$, the element $a$ has a right descent in $H$ and admits a reduced expression $a=s_1 \cdots s_r$ with $s_r \in H$. We can obtain a reduced expression for $u$ by deleting 2 appropriate letters in the reduced expression for $a$; one of these letters must be $s_r$ since $u\in W^H$. Evidently $s_1 \cdots s_{r-1}\in [u,a]=\{u,b,c,a\}$ and we may assume $b=s_1 \cdots s_{r-1}$. Since $b\notin W^H$, up to changing the chosen reduced expression of $a$, we may assume $s_{r-1} \in H$. Hence $u  =s_1 \cdots s_{r-2}$. Since $u  s_r \rhd u$, the Lifting Property applied to the multiplication on the right by $s_r$ implies $u  s_r \leq a$, and so $u  s_r=c$. 
The Lifting Property (applied to to the multiplications on the right by $s_r$ and $s_{r-1}$) also implies
$\{s_{r-1},s_r\}\subseteq D_R(d)$. 
Since $f$ cannot simultaneously be equal to $d  s_{r-1}$ and $d s_{r}$, by the Lifting Property $D_R(f)\cap \{s_{r-1},s_r\}\neq \emptyset$, and so $f$ cannot be in $W^H$.
\end{proof}

\begin{lem}
\label{noorbitaparabolico}
Let $(W,S)$ be a Coxeter system. Fix $H\subseteq S$, $w\in W^H$, and two $H$-special matchings $M,N$ of $w$.  Let $\mathcal O$ be an orbit of $\langle M,N \rangle $ with bottom element $\hat{0}_{\mathcal O}$. Then $\mathcal O \cap W^H$ can be either
\begin{enumerate}
\item the empty set, or
\item the entire orbit $\mathcal O$, or
\item the singleton $\{\hat{0}_{\mathcal O}\}$, or 
\item the chain $\mathcal C_{\hat{0}_{\mathcal O},u}$, for an appropriate coatom $u$ of the top element of $\mathcal O$,  
\end{enumerate}
where, for all $v\in \mathcal O$, we denote by $\mathcal C_{\hat{0}_{\mathcal O},v}$  the unique saturated chain from $\hat{0}_{\mathcal O}$ to $v$ such that, for any two of its elements $x,y\in \mathcal C_{\hat{0}_{\mathcal O},v}$ with $x\lhd y$, either $y=M(x)$, or $y=N(x)$.
\end{lem}
\begin{proof}
Suppose that $|\mathcal O \cap W^H|\notin  \{0,1, |\mathcal O| \}$. Let $u$ be a maximal element in $\mathcal O \cap W^H $. Clearly $u$ cannot be the bottom element of $\mathcal O$; by definition of $H$-special, $u$ cannot be the top element (otherwise $\mathcal O \cap W^H = \mathcal O$) and $\mathcal C_{0_{\mathcal O},u}$ is contained in $W^H$. It remains to show that $u$ is a coatom of the top element of $\mathcal O$. This follows from Lemma~\ref{noconfigurazione}.
\end{proof}

Following \cite[Subsection~4.1]{Mtrans}, for convenience' sake  we say that an $H$-special matching $M$ of $w$ \emph{calculates} the parabolic Kazhdan--Lusztig $R$-polynomials (or  is \emph{calculating}, for short) if, for all $u \in W^H$, $u \leq w$, we have
\begin{equation*}
 R_{u,w}^{H,x} (q)= \left\{ \begin{array}{ll}
R_{M(u),M(w)}^{H,x}(q), & \mbox{if $M(u)  \lhd u$,} \\
(q-1)R_{u,M(w)}^{H,x}(q)+qR_{M(u),M(w)}^{H,x}(q), & \mbox{if $M(u) \rhd u$
and $M(u) \in W^{H}$,} \\
(q-1-x)R_{u,M(w)}^{H,x}(q), & \mbox{if $M(u) \rhd u$ and  $M(u) \notin W^{H}$.} 
\end{array} \right. 
\end{equation*}
For  this definition, it is essential that the special matching be $H$-special. Note that all left multiplication matchings are calculating. 

The main result of \cite{M} is the following.
\begin{thm}
\label{computa}
Let $M$ be an $H$-special matching of $w$.
The parabolic Kazhdan--Lusztig $R$-polynomial $R_{u,w}^{H,x}(q)$ satisfies: 
\begin{equation}
 \label{calcola}
 R_{u,w}^{H,x} (q)= \left\{ \begin{array}{ll}
R_{M(u),M(w)}^{H,x}(q), & \mbox{if $M(u)  \lhd u$,} \\
(q-1)R_{u,M(w)}^{H,x}(q)+qR_{M(u),M(w)}^{H,x}(q), & \mbox{if $M(u) \rhd u$
and $M(u) \in W^{H}$,} \\
(q-1-x)R_{u,M(w)}^{H,x}(q), & \mbox{if $M(u) \rhd u$ and  $M(u) \notin W^{H}$.} 
\end{array} \right. 
\end{equation}
\end{thm}

By \cite[Theorem 7.7]{AHH}, the parabolic quotient $W^H$ is a pircon.
\begin{pro}
\label{hricetta}
Let $w\in W^H$. An $H$-special matching $M$ of $[e,w]$ gives rise to a special partial  matching $M^H$ of $[e,w]^H$, which is defined as follows:
\begin{equation}
 M^H(x)= \left\{ \begin{array}{ll}
M(x), & \mbox{if $M(x)  \in W^H $,} \\
x, & \mbox{if $M(x)  \notin W^H $,} 
\end{array} \right. 
\end{equation}
for all $x\in [e,w]^H$.
\end{pro}
\begin{proof}
The first two properties in Definition~\ref{accoppiamento parziale} are trivial. Let us prove the third  one. Let $x\lhd y$ and $M^H(x)\neq y$.  

Suppose that $M^H(y)\rhd y$. If $M^H(x)=x$ or $M^H(x)\lhd x$, then clearly  $M^H(x)<M^H(y)$ . If $M^H(x)\rhd x$ then $M^H(x)=M(x)\lhd M(y)=M^H(y)$ since $M$ is a special matching.

 Suppose that $M^H(y)\lhd y$. By the definition of a special matching,   $M(x)<x$ and   $M(x)\lhd M(y)$, and we are done since    $M(x)=M^H(x)$ and $ M(y)=M^H(y)$. 
  
    Suppose that $M^H(y)= y$. By the definition of $M^H$, we have  $y \lhd M(y) \notin W^H$. We are done if we show that it is not possible that $x\lhd M^H(x) \in W^H$. Indeed, this is not possible since, otherwise, $M(y)$ would cover both $y$ and  $M(x)$ by the definition of a special matching, but  an element outside $W^H$ has at most one coatom in $W^H$ (see \cite[Lemma 4.1]{Mtrans}). 
\end{proof}

Let $SPM^H\subseteq \bigcup_{w\in W^H\setminus\{e\}}SPM_w$ be the set of all special partial matchings that are obtained from $H$-special matchings by the recipe of Proposition~\ref{hricetta}: $$SPM^H=\{M^H \in SPM_w : M  \text{ is an $H$-special matching of  $w\in W^H\setminus\{e\}$}\}.$$
\begin{thm}
\label{quoziente sistema pircone}
If $W$ is a simply laced Coxeter group, then $(W^H, SPM^H)$ is a pircon system. Furthermore, for both $x=q$ and $x=-1$, the parabolic  Kazhdan--Lusztig $R^{H,x}$-polynomials  coincide with the Kazhdan--Lusztig $R^x$ polynomials of $(W^H, SPM^H)$ as a pircon. 
  \end{thm}  
\begin{proof}
Let us prove that $(W^H, SPM^H)$ is a pircon system. The only nontrivial property in Definition~\ref{sistema} is the last one. Fix $w\in W^H\setminus \{e\}$. Let $M^H$ and $N^H$ be two special partial matchings (obtained from the two special matchings $M$ and $N$)  such that $M^H(w)$ and $N^H(w)$ are defined and satisfy $M^H(w)\lhd w$ and $N^H(w)\lhd w$. For short, we denote  the restrictions of the matchings $M^H$ and $N^H$ with the same letters $M^H$ and $N^H$. (Actually, the restriction of a special partial matching obtained from an $H$-special matching $M$ is the special partial matching obtained from the $H$-special matching given by the restriction of $M$). 

If both $M$ and $N$ are left multiplication matchings, then all orbits of $\langle M,N \rangle $ in $[e,w]$ are dihedral intervals of the same rank $d$; furthermore, $\langle M^H, N^H \rangle (w)$ is a  dihedral interval of  rank $d$ and, by Lemma~\ref{noorbitaparabolico}, any other orbit  of $\langle M^H,N^H \rangle $ is either
\begin{itemize}
\item a dihedral orbit  of rank d, or
\item a chain-like orbit of rank 0, or
\item a chain-like orbit of rank $d-1$,
\end{itemize}
and $M^H$ and $N^H$ are strictly coherent.

If $M$ is not a left multiplication matching, then by \cite[Corollary~3.8]{Mtrans} there exists a  left multiplication
matching $L_M$ (which, being a left multiplication matching, is $H$-special) that commutes with $M$ and such that does not agree with $M$ on $w$. (We stress the fact that \cite[Corollary~3.8]{Mtrans} requires the hypothesis that $W$ is simply laced). Consider   $L_M^H$, the special partial matching of $[e,w]^H$ obtained from $L_M$. Notice that $L_M^H\in SPM^H$. All orbits of $\langle M,L_M \rangle $ in $[e,w]$ are dihedral intervals of rank 2 or 1; by Lemma~\ref{noorbitaparabolico}, an orbit which is a dihedral interval of rank 2 or 1 gives rise to an orbit of $\langle M^H,L_M^H \rangle $ which is either
\begin{itemize}
\item the empty set, or
\item a dihedral interval of rank 2 or 1, or
\item a chain of rank $1$ or $0$.
\end{itemize}
Since $M$ and $L_M$ do not agree  on $w$, the orbit $\langle M^H,L_M^H \rangle (w)$ is a dihedral interval of rank 2  and so $M^H$ and $L_M^H$ are strictly coherent.

If $N$ is a left multiplication matching, then the sequence $(M^H,L_M^H,N^H)$ is a sequence of strictly coherent special partial matchings. If $N$ is not a left multiplication matching, we apply the same argument and find a left multiplication matching $L_N$ whose associated  special partial matching  $L_N^H$  belongs to $SPM^H$ and is strictly coherent with $N^H$. If $L_M^H\neq L^H_N$, then the sequence $(M^H,L_M^H,L^H_N,N^H)$ is a sequence of strictly coherent special partial matchings. If $L_M^H = L^H_N$, then the sequence $(M^H,L_M^H=L^H_N,N^H)$ is a sequence of strictly coherent special partial matchings.

The last assertion now follows by Theorem~\ref{computa}.
\end{proof}

\begin{rem}
Theorem~\ref{quoziente sistema pircone} does not hold without the assumption that $W$ is simply laced. For example, if
\begin{itemize}
\item $W$ is a Coxeter group with generator set  $\{s, t,p\}$ and relations $m(s,t)\geq 5$, $m(s,p)=3$,  $m(t,p)\geq 3$,
\item $H=\{p\}$,
\item $w=ststps$,
\end{itemize}
then $w\in W^H$, and $[e,w]$ has exactly three $H$-special matchings. Since all of them send $w$ to $tstps$,  no two of the associated special partial matchings of $[e,w]^H$ are coherent.
Moreover, one can easily check that there are no other special partial matchings of $[e,w]^H$, and therefore  $[e,w]^H$ is not a dircon.
\begin{probl}
Let $(W,S)$ be a Coxeter system, $H\subseteq S$, and $w\in W^H$. When is  $[e,w]^H$ a dircon?
\end{probl}
Notice that, in general, not all special partial matchings of  $[e,w]^H$ come from $H$-special matchings of $[e,w]$ (even under the hypothesis $W$ is simply laced). 
\end{rem}

\bigskip

From a combinatorial point of view, the most intriguing conjecture about (classical) Kazhdan--Lusztig polynomials of Coxeter groups is the Combinatorial Invariance Conjecture of Kazhdan--Lusztig polynomials, which was independently formulated by Lusztig in private and by Dyer in \cite{Dyeth}.
\begin
{con}
\label{comb-inv-con}
Let $(W,S)$ be a Coxeter system, and $u,v\in W$. The (classical) Kazhdan--Lusztig polynomial $P_{u,v}(q)$ depends only on the isomorphism class of the interval $[u,v]$ as a poset. 
\end{con}
The Combinatorial Invariance Conjecture of Kazhdan--Lusztig polynomials is equivalent to the analogous conjecture on the combinatorial invariance of  Kazhdan--Lusztig $R$-polynomials, and is still very much open (see \cite{BCM1} for a partial result).
The problem of the combinatorial invariance of parabolic Kazhdan--Lusztig polynomials, which is clearly stronger than the combinatorial invariance of the ordinary Kazhdan--Lusztig polynomials, also attracted much attention but has only recently been found to be false (see \cite{Mtrans} for a counterexample by Mongelli and \cite{Mtrans,M} for more details on this subject).
For pircons, counterexamples already exist in small pircons: the two intervals $[d,a]$ and $[e,a]$ in the pircon of Example~\ref{nonvale} are both chains of length 3 and have different Kazhdan--Lusztig $R^x$-polynomials: 
$$R^x_{d,a}(q)=(q-1)(q-1-x)^2\neq  R^x_{e,a}(q)= (q-1)^2(q-1-x).$$

\section{Pircons and Stanley's kernels}
After introducing the Kazhdan--Lusztig $R$-polynomials of a refined pircon $(P,\mathcal M)$, it is natural to look for the analog of the Kazhdan--Lusztig polynomials in this general context. Since, for Coxeter groups, these are the Kazhdan--Lusztig--Stanley polynomials of the Kazhdan--Lusztig $R$-polynomials, one might want to define the Kazhdan--Lusztig polynomials of a refined pircon $(P,\mathcal M)$ as the Kazhdan--Lusztig--Stanley polynomials of the Kazhdan--Lusztig $R$-polynomials of $(P,\mathcal M)$. Unfortunately, in general, the Kazhdan--Lusztig--Stanley polynomials of the Kazhdan--Lusztig $R$-polynomials of a refined pircon do not exist.

Let us briefly recall from \cite{S} the definition of $P$-kernel and Kazhdan--Lusztig--Stanley polynomials.

Let $P$ be a locally finite graded poset, with rank function $\rho$. The \emph{incidence algebra} of $P$ over the polynomial ring  $\mathbb R[q]$, denoted
$I(P)$, is the associative algebra of functions $f$ assigning to each nonempty interval $[u, v]$ an element  $f_{u,v}(q)\in \mathbb R[q]$ (denoted also simply by $f_{u,v}$ when no confusion arises) with usual sum and convolution product: $(f+g)_{u,v}=f_{u,v}+g_{u,v}$ and $(f \cdot g)_{u,v}=\sum_{z:u\leq z\leq v} f_{u,z} \;g_{z,v}$, for all $f,g \in  I(P)$ and all $u,v\in P$ with $u\leq v$. The  identity element of  $I(P)$ is the \emph{delta function} $\delta$, defined by $\delta_{u,v}=\left\{\begin{array}{ll}
1 & \text{if $u=v$,} \\
0 & \text{if $u<v$.}
\end{array} \right.$
An element $f\in I(P)$ is invertible if and only if $f_{u,u} \in \mathbb R \setminus \{0\}$ for all $u\in P$. By convention, for all $f\in I(P)$, we let  $f_{u,v}= 0$ whenever $u\not\leq v$. We say that $f\in I(P)$ is unitary if $f_{u,u}=1$, for all $u\in P$.
Let 
\begin{itemize}
\item $I'(P)=\{f \in I(P): \deg f_{u,v} \leq \rho(v)-\rho(u), \text{ for all $u,v\in P$ with $u\leq v$}\}$, 
\item  $I_{\frac{1}{2}}(P)=\{f\in I'(P) \text{ unitary}: \deg f_{u,v}<\frac{1}{2}(\rho(v)-\rho(u)), \text{ for all  $u,v\in P$ with $u< v$}\}$.
\end{itemize}
Note that $I'(P)$ is a subalgebra of $I(P)$, closed under taking inverse. 
Given $f \in I'(P)$,  we denote by $\overline{f}$ the element of $I'(P)$  such that $\overline{f}_{u,v}(q)=q^{\rho(v)-\rho(u)} f_{u,v}(q^{-1})$, for all  $u,v\in P$ with $u\leq v$. 
Notice that the map $\; \bar{\text{}} \;$ is an involution on $I'(P)$.

A unitary element $K\in I(P)$ is a \emph{$P$-kernel} if there exists
an invertible element $f \in I(P)$ such that $K \cdot f= \overline{f}$.
Such an  element $f\in I(P)$  is called \emph{invertible $K$-totally acceptable} function in \cite{S}.  See \cite[Theorem~6.5, Proposition~6.3, Corollary~6.7]{S} for a proof of the next result.
\begin{thm}
Let $P$ be a locally finite graded poset.
\begin{enumerate}
\item A unitary $K\in I'(P)$  is a $P$-kernel if and only if $K \cdot \overline{K}=\delta$.
\item There is a bijection from the set of $P$-kernels of $I'(P)$ to $ I_{\frac{1}{2}}(P)$ that assigns to $K$ an invertible $K$-totally acceptable function.
\end{enumerate} 
\end{thm} 

Let $K\in I'(P)$ be a $P$-kernel. Following \cite{BJoA}, we refer to the unique invertible $K$-totally acceptable function of $ I_{\frac{1}{2}}(P)$ as the \emph{Kazhdan--Lusztig--Stanley polynomials} of $K$. 

As shown in \cite[Sections 6 and 7]{S}, Kazhdan--Lusztig--Stanley polynomials specialize to many interesting objects. As an example, the  Kazhdan--Lusztig $R$-polynomials of a Coxeter group $W$ form a $W$-kernel whose Kazhdan--Lusztig--Stanley polynomials are the Kazhdan--Lusztig polynomials of $W$. More generally,  for $H\subseteq S$ and $x\in\{q,-1\}$, the parabolic Kazhdan--Lusztig $R^{H,x}$-polynomials of $W^H$ form a $W^H$-kernel whose Kazhdan--Lusztig--Stanley polynomials are the parabolic Kazhdan--Lusztig $P^{H,x}$-polynomials of $W^H$ (see \cite[Lemma~2.8~(iv) and Proposition~3.1]{Deo87}).

Unfortunately, as shown in Example~\ref{nonvale}, in general the Kazhdan--Lusztig $R^x$-polynomials of a refined pircon $P$ do not form a $P$-kernel. Therefore, there are no Kazhdan--Lusztig--Stanley polynomials associated with them.

\begin{exa}
\label{nonvale}
Let $P$ be the pircon in the first picture of Figure~\ref{nonkernel}, with the structure of a refined pircon given by the special partial matchings depicted in the other pictures. 

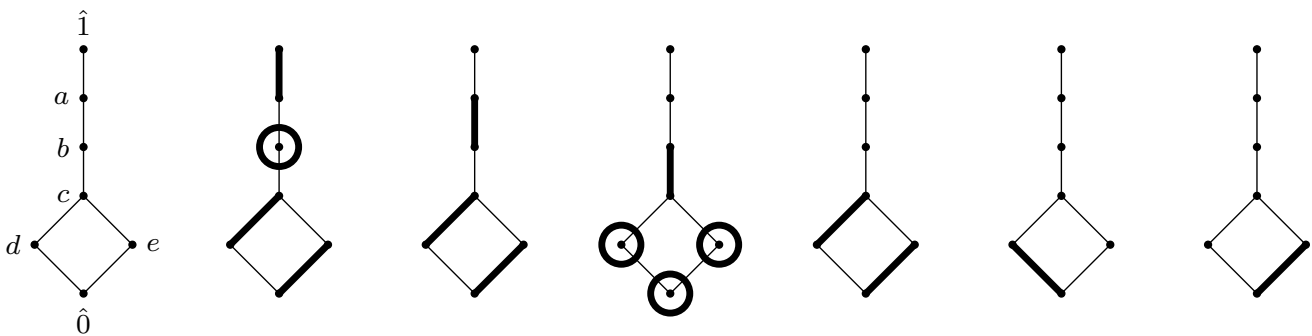
\begin{figure}[h]
\begin{center}
$$
\scalebox{1.3}{
\begin{tikzpicture}
\foreach \x in {0,2,4,6,8,10,-2}
\draw[fill=black]{(0+\x,0) circle(1pt)};
\foreach \x in {0,2,4,6,8,10,-2}
\draw[fill=black]{(-0.5+\x,0.5) circle(1pt)};
\foreach \x in {0,2,4,6,8,10,-2}
\draw[fill=black]{(0+\x,1) circle(1pt)};
\foreach \x in {0,2,4,6,8,10,-2}
\draw[fill=black]{(0+\x,1.5) circle(1pt)};
\foreach \x in {0,2,4,6,8,10,-2}
\draw[fill=black]{(0.5+\x,0.5) circle(1pt)};
\foreach \x in {0,2,4,6,8,10,-2}
\draw[fill=black]{(0+\x,2) circle(1pt)};
\foreach \x in {0,2,4,6,8,10,-2}
\draw[fill=black]{(0+\x,2.5) circle(1pt)};
\foreach \x in {0,2,4,6,8,10,-2}
 \draw{(0+\x,0)--(0.5+\x,0.5)};
 \foreach \x in {0,2,4,6,8,10,-2}
 \draw{(0+\x,0)--(-0.5+\x,0.5)}; 
 \foreach \x in {0,2,4,6,8,10,-2}
 \draw{(-0.5+\x,0.5)--(0+\x,1)}; 
 \foreach \x in {0,2,4,6,8,10,-2}
 \draw{(0.5+\x,0.5)--(0+\x,1)}; 
 \foreach \x in {0,2,4,6,8,10,-2}
 \draw{(0+\x,1)--(0+\x,1.5)}; 
 \foreach \x in {0,2,4,6,8,10,-2}
\draw{(0+\x,2)--(0+\x,1.5)}; 
\foreach \x in {0,2,4,6,8,10,-2}
 \draw{(0+\x,2)--(0+\x,2.5)}; 
 
 \node[above]at(-2,2.5){\tiny{$\hat{1}$}};
  \node[left]at(-2,2){\tiny{$a$}};
  \node[left]at(-2,1.5){\tiny{$b$}};
  \node[left]at(-2,1){\tiny{$c$}};
  \node[left]at(-2.5,0.5){\tiny{$d$}};
  \node[right]at(-1.5,0.5){\tiny{$e$}};
  \node[below]at(-2,0){\tiny{$\hat{0}$}};
 
 \draw[ line width=2pt]{(-0.5,0.5)--(0,1)};
 \draw[ line width=2pt]{(0,0)--(0.5,0.5)}; 
 \draw[ line width=2pt]{(0,2)--(0,2.5)}; 
 \draw[ line width=2pt]{(0,1.5) circle (.2)};
 
 \draw[ line width=2pt]{(1.5,0.5)--(2,1)};
 \draw[ line width=2pt]{(2,0)--(2.5,0.5)}; 
 \draw[ line width=2pt]{(2,2)--(2,1.5)}; 
 
 \draw[ line width=2pt]{(4,1)--(4,1.5)}; 
 \draw[ line width=2pt]{(3.5,0.5) circle (.2)};
 \draw[ line width=2pt]{(4,0) circle (.2)};
  \draw[line width=2pt]{(4.5,0.5) circle (.2)};
  
  \draw[ line width=2pt]{(5.5,0.5)--(6,1)};
 \draw[ line width=2pt]{(6,0)--(6.5,0.5)}; 
 
  \draw[ line width=2pt]{(8,0)--(7.5,0.5)}; 
  
   \draw[ line width=2pt]{(10,0)--(10.5,0.5)}; 
  
\end{tikzpicture}
}$$
\end{center}
\caption{\label{nonkernel} A refined pircon.} 
\end{figure}

Using the recursive formula (\ref{klperpirconi}), we can compute all Kazhdan--Lusztig $R^x$-polynomials of this refined pircon. In particular, 
$$ (R^x\cdot \overline{R^x})_{e,\hat{1}}= \sum_{z: e\leq z}R^x_{e,z}(q) \; q^{5-\rho(z)} \;R^x_{z,\hat{1}}(\frac{1}{q})=-q(q-1)^2\neq 0,$$
 so  the $R^x$-polynomials of this refined pircon do not form a $P$-kernel.
\end{exa}

{\bf Acknowledgements:}  I am grateful  to  the anonymous referee for the valuable suggestions, which provided insights that helped improve the paper.

\end{document}